\documentclass[11pt]{amsart}

\usepackage{latexsym}
\usepackage{amsmath,amsthm,amssymb, amscd,color,mathtools,bbm}
\usepackage{mathrsfs}
\usepackage[all]{xy}
\usepackage[all]{xy}
\usepackage{tikz}
\usepackage{tikz-cd}
\usepackage{adjustbox}
\usetikzlibrary{decorations.pathmorphing}
\usepackage{blkarray} 
\usepackage{framed}
\theoremstyle{plain}
\newtheorem{theorem}{Theorem}[section]
\newtheorem{lemma}[theorem]{Lemma}
\newtheorem{corollary}[theorem]{Corollary}
\numberwithin{equation}{section}

\newtheorem{innercustomgeneric}{\customgenericname}
\providecommand{\customgenericname}{}
\newcommand{\newcustomtheorem}[2]{%
	\newenvironment{#1}[1]
	{%
		\renewcommand\customgenericname{#2}%
		\renewcommand\theinnercustomgeneric{##1}%
		\innercustomgeneric
	}
	{\endinnercustomgeneric}
}

\newcustomtheorem{customthm}{Theorem}

\theoremstyle{definition}
\newtheorem{definition}[theorem]{Definition}
\newtheorem{example}[theorem]{Example}

\newtheorem{proposition}[theorem]{Proposition}
\newtheorem{remark}[theorem]{Remark}
\newtheorem{conj}[theorem]{Conjecture}

\theoremstyle{remark}

\newcommand{\CC}{\mathbb{C}}
\newcommand{\QQ}{\mathbb{Q}}

\newcommand{\ZZ}{\mathbb{Z}}

\newcommand{\WG}{\mbox{WGr}}
\newcommand{\G}{\mbox{Gr}}
\newcommand{\WW}{\mathbb{W}}
\newcommand{\Pl}{ Pl\"{u}cker}

\usepackage{blkarray} 
\usepackage{framed}

\newcustomtheorem{customlemma}{Lemma}
\newcustomtheorem{customprop}{Proposition}

\newcommand{\lcm}{\operatorname{lcm}}

\begin{document}

\title[Integral cohomology of weighted Grassmann orbifold] {Integral cohomology rings of weighted Grassmann orbifolds and rigidity properties}

\author[K Brahma]{Koushik Brahma}
\address{Department of Pure and Applied Mathematics, Waseda University 3-4-1 Okubo, Shinjuku-ku,
Tokyo 169-8555, Japan}
\email{koushikbrahma95@gmail.com, w.iac24166@kurenai.waseda.jp}



\date{\today}
\dedicatory{}

\abstract In this paper, we introduce `Pl\"{u}cker weight vector' and establish the definition of a weighted Grassmann orbifold ${\rm Gr}_{\bf b}(k,n)$, corresponding to a  Pl\"{u}cker weight vector `${\bf b}$'. We achieve an explicit classification of weighted Grassmann orbifolds up to certain homeomorphism in terms of the Pl\"{u}cker weight vectors. We study the integral cohomology of ${\rm Gr}_{\bf b}(k,n)$ and provide some sufficient conditions such that the integral cohomology of ${\rm Gr}_{\bf b}(k,n)$ has no torsion. We explicitly describe the formula of the equivariant structure constants with respect to the equivariant Schubert basis in equivariant cohomology ring of divisive weighted Grassmann orbifolds with integer coefficients. Eminently, we compute the integral cohomology rings of divisive weighted Grassmann orbifolds explicitly.
\endabstract

\maketitle

\section{Introduction}
Grassmann manifolds are the central objects of study in algebraic topology, algebraic geometry, and enumerative geometry, see \cite{JP,KnTa,Lak}. 
Consider the $n$-dimensional complex vector space $\CC^n$ and a positive integer $k$ satisfying $1\leq k<n$. The space comprising all $k$-dimensional subspaces in $\CC^n$ possesses the structure of a smooth manifold. This manifold is known as a complex Grassmann manifold (in short, Grassmann manifold in this paper) and is denoted by $\G(k,n)$. 
Grassmann manifold $\G(k,n)$ can be equivariantly embedded into the complex projective space $\CC P^m$, where $m+1={n \choose k}$. This embedding is known as the Pl\"{u}cker embedding, presenting $\G(k,n)$ as a projective algebraic variety. The image of the Grassmann manifold in the projective space is described by a set of homogeneous polynomial equations known as Pl\"{u}cker relations. The non-zero points in $\CC^{m+1}$ that satisfy Pl\"{u}cker relations are referred to as Pl\"{u}cker coordinates.

In \cite{Ka}, Kawasaki introduced a ${\bf{b}}$-action of $\CC^*$ on $\CC^{m+1}\setminus\{\mathbf{0}\}$ for a weight vector ${\bf{b}}=(b_0,\dots,b_m)\in(\ZZ_{\geq 1})^{m+1}$ as follows:
\begin{equation}\label{eq_weigh_act}
	t(z_0,z_1,z_2,\dots,z_m)=(t^{b_0}z_0,t^{b_1}z_1,t^{b_2}z_2,\dots,t^{b_m}z_m).
\end{equation}

The weighted projective space $\WW P(b_0,b_1,\dots,b_m)$ is defined as the orbit space of the ${\bf{b}}$-action of $\CC^{*}$ on $\CC^{m+1}\setminus\{\mathbf{0}\}$. Notably, when ${\bf{b}}$ is chosen as $(1,1,\dots,1)$, it becomes the complex projective space $\CC P^m$. In the article \cite{Ka}, Kawasaki proved that the integral cohomologies of weighted projective spaces have no torsion and are concentrated in even degrees. Moreover, he computed the integral cohomology ring of weighted projective spaces explicitly. In \cite{HHRW}, Harada, Holm, Ray, and Williams computed the integral generalized equivariant cohomology ring of a divisive weighted projective space. Several cohomological properties of weighted projective space can be found in \cite{Am}, and \cite{BFR}.

Corti and Reid introduced weighted Grassmannian in \cite{CoRe} as the weighted projective analog of the Grassmann manifold.  For any two positive integers $ k<n$, Abe and Matsumura defined weighted Grassmannian $\WG(k,n)$ in \cite{AbMa1} corresponding to a weight vector $W=(w_0,w_1,\dots,w_n)\in (\ZZ_{\geq 0})^n$ and $a\in \ZZ_{\geq 1}$. They studied the equivariant cohomology ring of $\WG (k,n)$ with rational coefficients. More generally, the equivariant cohomology ring of weighted flag varieties with rational coefficients was studied in \cite{ANQ}. Hilbert series of some weighted flag varieties was studied in \cite{QS,Qu}. The author and Sarkar \cite{BS} gave another topological definition of $\WG(k,n)$ and called them weighted Grassmann orbifolds. It has an orbifold structure and admits a $q$-CW complex structure similar to the Schubert cell decomposition of the Grassmann manifold.  For basic properties of orbifolds, readers are referred to \cite{ALR}.

In this article, we delve into the concept of Pl\"{u}cker coordinates, which are defined as solutions to a system of homogeneous polynomial equations known as Pl\"{u}cker relations, as illustrated in \eqref{eq_pl_vect}. This perspective naturally leads us to introduce Pl\"{u}cker weight vector ${\bf{b}}=(b_0,b_1,\dots,b_m)\in (\ZZ_{\geq 0})^{m+1}$, where $m+1={n \choose k}$ for two positive integers $k$ and $n$ such that $k<n$. Associated to each Pl\"{u}cker weight vector ${\bf{b}}$, we define a topological space $\G_{\bf{b}}(k,n)$. The space $\G_{\bf b}(k,n)$ has an orbifold structure and call this a weighted Grassmann orbifold. 
We introduce the notion of Pl\"{u}cker permutation, which yields several invariant $q$-CW complex structures on $\G_{\bf b}(k,n)$. Moreover, we classify weighted Grassmann orbifolds up to certain homeomorphism in terms of Pl\"{u}cker weight vectors and Pl\"{u}cker permutations.

The main goal of this paper is to compute the integral cohomology ring of weighted Grassmann orbifolds with integer coefficients. The author and Sarkar \cite{BS} introduced 
divisive weighted Grassmann orbifolds and computed the generalized $T^n$-equivariant cohomology ring of these spaces with integer coefficients, where $T^n$ is the standard compact abelian torus of dimension $n$. There exists an equivariant Schubert basis for the equivariant cohomology ring divisive weighted Grassmann orbifolds as $H^{*}_{T^n}(\{pt\};\ZZ)$ module. In this article, we explicitly calculate the equivariant structure constants with respect to the equivariant Schubert basis in $H_{T^n}^{*}(\G_{\bf b}(k,n);\ZZ)$ with integer coefficients. This paves the way for computing the cup product for the cohomology ring of a divisive weighted Grassmann orbifold with integer coefficients.  The paper is organized as follows.

 In Section \ref{bld_seq_wgt_gmn_obd}, we explore the reltionship between the Pl\"{u}cker relations and Pl\"{u}cker coordinates. We denote the set of all  Pl\"{u}cker coordinates by $Pl(k,n)$ as a subset of $\CC^{m+1}\setminus \{\mathbf{0}\}$, where $m+1={n \choose k}$. The Pl\"{u}cker relations inspired us to introduce Pl\"{u}cker weight vector ${\bf{b}}=(b_0,b_1,\dots,b_m)\in (\ZZ_{\geq 1})^{m+1}$ in the manner that $Pl(k,n)$ is invariant subset of $\CC^{m+1}\setminus\{\mathbf{0}\}$  with respect to the ${\bf{b}}$-action of $\CC^{*}$ defined in \eqref{eq_weigh_act}, where $m+1={n \choose k}$. Then we define a topological space $\G_{\bf{b}}(k,n)$, as the orbit space of $Pl(k,n)$ by the ${\bf b}$-action of $\CC^*$ for a Pl\"{u}cker weight vector ${\bf b}=(b_0,b_1,\dots,b_m)$, see Definition \ref{def_wgt_gsm}. We prove that this class of $\G_{\bf b}(k,n)$ broadens the class of weighted Grasmann orbifold discussed in \cite{AbMa1} and \cite{BS}; see Proposition \ref{prop_rel_two_wt}. We discuss the orbifold and $q$-CW complex structure of $\G_{\bf{b}}(k,n)$ and call it a weighted Grassmann orbifold. The $q$-CW complex structure of $\G_{\bf{b}}(k,n)$ induces a building sequence
 \[\{pt\}=X^0 \subset X^1 \subset X^2 \subset \dots \subset X^m=\G_{\bf b}(k,n)\]
 of $\G_{\bf{b}}(k,n)$, see \eqref{filtration}. This induces a cofibration 
$$L'(b_i;{\bf{b}}^{(i)}) \hookrightarrow X^{i-1} \to X^i,$$
see Proposition \ref{prop_cofibre},
 where $L'(b_i;{\bf{b}}^{(i)})$ is a lens complex defined in \cite{Ka}. 

In Section \ref{sec_clss_prob}, we introduce Pl\"{u}cker permutations, permutations $\sigma$ on $\{0,1,\dots,m\}$  those act on the Pl\"{u}cker coordinates. Then we show that all Pl\"{u}cker weight vectors are invariant under Pl\"{u}cker permutations.  Moreover, Pl\"{u}cker permutation provides several different $q$-CW complex structures of a weighted Grassmann orbifold. We prove the following:
\begin{customthm}{A}[Theorem \ref{thm_cl_wgt_gsm}]
  Two weighted Grassmann orbifolds are weakly equivariantly homeomorphic if their Pl\"{u}cker weight vectors differ by a Pl\"{u}cker permutation and a scalar multiplication.  
\end{customthm}
Conversely, we formulate a conjecture asserting that weighted Grassmann orbifolds can be classified up to homeomorphism solely in terms of their Pl\"{u}cker weight vectors together with Pl\"{u}cker permutations. In support of this conjecture, we prove Lemma \ref{lem_nor_vec}, Lemma \ref{lem_prim} and Lemma \ref{lemma_int_coh}. 
Moreover, we prove the following regidity theorem.
\begin{customthm}{B}[Theorem \ref{thm_cohom_rig}]
If two weighted Grassmann orbifolds ${\rm{Gr}}_{\bf b}(k,n)$ and ${\rm{Gr}}_{\bf c}(k,n)$  are homeomorphic such that each coordinate element mapped to a coordinate element then ${\bf b}$ and ${\bf c}$ are identical up to a scalar multiplication and a permutation.
\end{customthm}
The coordinate elements are the set of all fixed points for the several group actions on $\G_{\bf b}(k,n)$. Thus, the above Theorem classifies the weighted Grassmann orbifolds up to certain weakly equivariant homeomorphism, see Corollary \ref{cor_cls_eq_hom}.

 In Section \ref{sec_tor_int_coh_wgt_gsm}, we show that
  for a prime $p$, the integral cohomology of any weighted Grassmann orbifold contains no $p$-torsion under some hypothesis on the Pl\"{u}cker weighted vector, see Theorem \ref{thm_no_p_tor}. 
  We study some non-trivial examples of $\G_{\bf{b}}(2,4)$ and $\G_{\bf{b}}(2,5)$ such that the integral cohomologies of those weighted Grassmann orbifolds have no torsion and concentrated in even degrees. We define divisive weighted Grassmann orbifolds in terms of Pl\"{u}cker permutations and Pl\"{u}cker weight vectors. 
 Then, we show that the integral cohomologies of divisive Grassmann orbifolds are torsion-free and is concentrated in even degrees.

In Section \ref{sec_st_con_div_gsm_orb}, we describe the $T^n$-equivariant cohomology ring of the divisive $\G_{\bf{b}}(k,n)$  with integer coefficients and the equivariant Schubert basis $\{{\bf{b}}\widetilde{\mathbb{S}}_{\lambda^i}\}_{i=0}^m$ for the aforementioned ring as a $\ZZ[y_1,\dots,y_n]$ module. The main achievement of this section is computing the equivariant structure constants ${\bf{b}}\widetilde{\mathscr{C}}_{i~j}^{\ell}$ with respect to the basis $\{{\bf{b}}\widetilde{\mathbb{S}}_{\lambda^i}\}_{i=0}^m$ satisfies the following.     
\begin{equation}\label{eq_wgt_mul1}
	{\bf{b}}\widetilde{\mathbb{S}}_{\lambda^i}{\bf{b}}\widetilde{\mathbb{S}}_{\lambda^j}=\sum_{\ell=0}^m {\bf{b}}\widetilde{\mathscr{C}}_{i~j}^{\ell}{\bf{b}}\widetilde{\mathbb{S}}_{\lambda^\ell}.
\end{equation} 
If ${\bf{b}}=(1,\dots,1)$, the above basis leads to the Schubert basis $\{\widetilde{\mathbb{S}}_{\lambda^i}\}_{i=0}^m$ for $H_{T^n}^{*}(\G(k,n);\mathbb{Z})$. We first describe the equivariant structure constants $\widetilde{\mathscr{C}}_{i~j}^{\ell}$ with respect to the basis $\{\widetilde{\mathbb{S}}_{\lambda^i}\}_{i=0}^m$ in \eqref{thm_eq_st_con} following \cite{KnTa}.
Then using Vietoris-Begle mapping theorem from \cite{Sp} we obtain a basis $\{P\widetilde{\mathbb{S}}_{\lambda^i}\}_{i=0}^m$ for  $H_{T^{n+1}}^{*}(Pl(k,n);\mathbb{Q})$ as a $\QQ[y_1,y_2,\dots,y_n]$ module, 
Then we deduce that ${\bf{b}}\widetilde{\mathbb{S}}_{\lambda^i}$ can be obtained by some isomorphic image of $P\widetilde{\mathbb{S}}_{\lambda^i}$. Eventually, we compute the equivariant structure constants with respect to the basis $\{{\bf{b}}\widetilde{\mathbb{S}}_{\lambda^i}\}_{i=0}^m$.
\begin{customthm}{C}[Theorem \ref{thm_wei_st_con}]
The equivariant structure constant ${\bf{b}}\widetilde{\mathscr{C}}_{i~j}^{\ell}$ is given by $${\bf{b}}\widetilde{\mathscr{C}}_{i~j}^{\ell}=\sum_{\lambda^\ell\succeq\lambda^q\succeq \lambda^i,\lambda^j}\sum_{R}\sum_{s=0}^{|R|}K(i,j,q,R)\mathcal{L}_{s,R} {\bf{b}}\widetilde{\mathscr{C}}_{q~g^s}^{\ell}.$$
\end{customthm}

Here, $R$ is a finite collection of elements in $\{1,2,\dots,n-1\}$. The non-negative integers $K(i,j,q,R)$ are known due to \cite{KnTa}, while the explicit formulae for the remaining coefficients $\mathcal{L}_{s,R} $ and ${\bf{b}}\widetilde{\mathscr{C}}_{q~g^s}^{\ell}$ are explained in \eqref{eq_Lsqr} and \eqref{eq_rep_wgt_per_rul} respectively. As a corollary, We describe the equivariant Pieri rule in $H_{T^n}^{*}(\G_{\bf b}(k,n);\ZZ)$. 
We demonstrate the positivity of equivariant structure constants; see Theorem \ref{thm_pos_eq_st_con}.
Using the upper triangularity of every basis element ${\bf{b}}\widetilde{\mathbb{S}}_{\lambda^i}$, we prove that:
\begin{customthm}{D}[Theorem \ref{prop_st_con_int}]
    Let ${\rm Gr}_{\bf{b}}(k,n)$ be a divisive weighted Grassmann orbifold. Then all the equivariant structure constants ${\bf{b}}\widetilde{\mathscr{C}}_{i~j}^{\ell}$ are in $ \ZZ[y_1,y_2,\dots,y_n]$.
\end{customthm}

In Section \ref{sec_int_coh}, we work out on a basis $\{{\bf b}\mathbb{S}_{\lambda^i}\}_{i=0}^m$ for the integral cohomology groups of divisive weighted Grassmann orbifolds. Finally, we compute the structure constant ${\bf{b}}\mathscr{C}_{i~j}^{\ell}$ with respect to the above basis with integer coefficients:
\begin{customthm}{E}[Theorem \ref{thm_st_con_ord_coh}]
    The structure constant ${\bf{b}}\mathscr{C}_{i~j}^{\ell}$ is given by $${\bf{b}}\mathscr{C}_{i~j}^{\ell}=\mathscr{C}_{i~j}^{\ell}+\sum_{\lambda^q~:~\lambda^\ell\succ\lambda^q\succeq\lambda^i,\lambda^j}\mathscr{D}_{i~j}^{\ell~q}.$$
\end{customthm}

Here $\mathscr{C}_{i~j}^{\ell}$ is the structure constant in $H^{*}(\G(k,n);\ZZ)$ and $\mathscr{D}_{i~j}^{\ell~q}$ is described in \eqref{eq_dij}.
This explicitly describes the integral cohomology ring of a divisive $\G_{\bf{b}}(k,n)$. We discuss an example to describe the explicit computation of the structure constants and the integral cohomology ring of a divisive $\G_{\bf{b}}(2,4)$. 

In appendix, we explore the integral cohomology of $\G_{\bf b}(2,4)$.
The main result of this section is the following:
 \begin{customthm}{F}[Theorem \ref{tor_in_2,4}]
$H^i({\rm{Gr}}_{\bf{b}}(2,4);\ZZ)$ is torsion-free and concentrated in even degrees for any weighted Grassmann orbifold ${\rm{Gr}}_{\bf b}(2,4)$ if $i\neq 3$. 
 \end{customthm}

\section{Building sequences of weighted Grassmann orbifolds} \label{bld_seq_wgt_gmn_obd}
In this Section, we introduce the notion of the Pl\"{u}cker weight vector, which we use to give a new definition of the weighted Grassmann orbifold. Then, we show that this definition generalizes the weighted Grassmannian discussed in \cite{AbMa1} and \cite{BS}. We discuss a $q$-CW complex structure of weighted Grassmann orbifold and explore the Lens complexes those arise in the cofibration induced  from the $q$-CW complex structure.

For two positive integers $k$ and $n$ such that $k<n$, we recall a Schubert symbol $\lambda$ is a sequence of $k$ integers $(\lambda_1,\lambda_2,\dots, \lambda_k)$ such that $1\leq \lambda_1 < \lambda_2 < \cdots < \lambda_k\leq n$. 
 A pair $(s,s')$ is called a inversion of $\lambda$ if $s\in \lambda, s'\notin \lambda$ and $s<s'$. The set of all inversions of $\lambda$ is denoted by $\rm{inv}(\lambda)$. If $(s,s') \in \rm{inv}(\lambda)$, then $(s,s')\lambda$ denote the Schubert symbol obtained by
replacing $s$ by $s'$ in $\lambda$ and ordering the later set.

Define a partial order `$\preceq$' on the set of Schubert symbols for $k<n$ by
\begin{equation}\label{eq_bru_ord}
	\lambda \preceq \mu \text{ if } \lambda_i\leq \mu_i \text{ for all } i=1,2,\dots,k.
\end{equation}
The Schubert symbols for $k<n$ form a lattice with respect to this partial order `$\preceq$'. 

There are $n\choose k$ many Schubert symbols for two positive integers $k$ and $n$ such that $k<n$. Now consider a total order $`<$'
on the set of all Schubert symbols which preserve the partial order `$\prec$'. In other words, if $\lambda\prec\mu$ then $\lambda<\mu$. For example, one can consider the dictionary order on the set of Schubert symbols. Then, it is a total order that preserves the partial order $\preceq$ defined in \eqref{eq_bru_ord}. Thus we can consider a chain \begin{equation}\label{eq_total_ord}
	\lambda^0<\lambda^1<\lambda^2< \dots <\lambda^m
\end{equation} 
on Schubert symbols for $k<n$, where $m = {n \choose k} -1$.

\subsection{Pl\"{u}cker weight vectors and weighted Grassmann orbifolds}\label{subsec_pl_wgt_vec}
Let $\Lambda^k(\CC^n)$ be the $k$-th exterior product of the $n$-dimensional complex vector space $\CC^n$. Note that the standard ordered basis $\{e_1, e_2, \dots, e_n\}$ of $\mathbb{C}^n$ induces an ordered basis $\{e_{\lambda^0}, e_{\lambda^1}, \dots, e_{\lambda^m}\}$ of $\Lambda^k(\mathbb{C}^n)$, where $e_{\lambda^i}=e_{\lambda^i_1}\wedge e_{\lambda^i_2}\wedge\dots\wedge e_{\lambda^i_k}$ for a Schubert symbol $\lambda^i=(\lambda^i_1,\lambda^i_2,\dots,\lambda^i_k)$. Then one can identify $\Lambda^k(\mathbb{C}^n)$ with $\mathbb{C}^{m+1}$. 
Consider a subset of $\Lambda^k(\CC^n)$ defined by the following
\begin{equation}\label{eq_pl_kn}
   Pl(k,n):= \{a_1\wedge a_2\wedge a_3\wedge\dots\wedge a_k\colon a_i\in \CC^n\}\setminus\{\bf{0}\}. 
\end{equation}

 Then $Pl(k,n)\subseteq \CC^{m+1}\setminus\{\bf{0}\}$. Consider an element $z=\sum_{i=0}^m z_ie_{\lambda^i}\in \CC^{m+1}\setminus\{\bf{0}\}$. Then $z\in Pl(k,n)$ iff it satisfies the following condition
\begin{equation}\label{eq_plu_coord}
  \sum_{j=0}^{k}(-1)^jz_{r_j}z_{s_j}=0 , 
\end{equation}
where 
$\lambda^{r_j}=(i_1,\ldots,i_{k-1},\ell_j)$ and $\lambda^{s_j}=(\ell_0,\ldots,\widehat{\ell}_j,...,\ell_k)$ for any two ordered sequences $1\leq i_1<\dots<i_{k-1}\leq n$ and $1\leq \ell_0<\ell_1<\dots<\ell_k\leq n$, see \cite[Theorem 3.4.11]{Jac}. 

The relation described in \eqref{eq_plu_coord} is known as a Pl\"{u}cker relation. The coordinates of $Pl(k,n)$ are known as Pl\"{u}cker coordinates. Note that a Pl\"{u}cker relation $\mathfrak{p}$ as in \eqref{eq_plu_coord} uniquely corresponds to $(k+1)$ many pairs $\{(r_0^\mathfrak{p},s_0^\mathfrak{p}), (r_1^\mathfrak{p},s_1^\mathfrak{p}), \dots, (r_k^\mathfrak{p},s_k^\mathfrak{p})\}$ of elements in $\{0,1,\dots,m\}$.

\begin{remark}
In the definition of {\Pl} relation, we can assume $2k\leq n$. If $2k>n$ consider $k'=n-k$, then $2k'=2n-2k<n$. Note that ${n\choose k}={n \choose k'}$. Let $\lambda=(\lambda_1,\dots,\lambda_k)$ be a Schubert symbol for $k<n$. Then $\{1,2,\dots,n\}\setminus \{\lambda_1,\dots,\lambda_k\}$ in an increasing order gives a Schubert symbol for $k'<n$. This correspondence of the Schubert symbol gives a relation as in \eqref{eq_plu_coord} in $Pl(k,n)$ induced from $Pl(k',n)$.
\end{remark}

 Consider the ${\bf{b}}$-action of $\CC^*$ on $\CC^{m+1}\setminus\{{\bf 0}\}$ defined in \eqref{eq_weigh_act}.
 Now the natural question is whether $Pl(k,n)$ is an invariant subset of $\CC^{m+1}\setminus\{{\bf 0}\}$ with respect to this ${\bf{b}}$-action. Note that $Pl(k,n)$ will be an invariant subset of $\CC^{m+1}\setminus\{{\bf 0}\}$ iff   $$(t^{b_0}z_0,t^{b_1}z_1,,\dots,t^{b_m}z_m)\in Pl(k,n) \text{ for every }t \in \CC^{*} \text{ and } (z_0,z_1,\dots,z_m)\in Pl(k,n).$$ 
Using \eqref{eq_plu_coord}, $(t^{b_0}z_0,t^{b_1}z_1,\dots,t^{b_m}z_m)\in Pl(k,n)$ iff 
\begin{equation}\label{eq_pl_vect}
   \sum_{j=0}^{k}(-1)^j t^{b_{r_j}+b_{s_j}}z_{r_j}z_{s_j}=0. 
\end{equation} 
 This may not be true in general for any arbitrary ${\bf{b}}=(b_0, b_1,\dots, b_m)\in (\mathbb{Z}_{\geq 1})^{m+1}$.

\begin{remark}\label{rmk_pl_rln}
Consider $(k+1)$ pairs $(r_0,s_0), \dots, (r_k,s_k)$ corresponding to a {\Pl} relation. 
We fix two pairs $(r_i,s_i)$, $(r_j,s_j)$ for $i\neq j$ and consider $z=(z_0,z_1,\dots,z_m)\in \CC^{m+1}$ by $z_{r_i}=z_{s_i}=z_{r_j}=1$, $z_{s_j}=(-1)^{i+j+1}$ and the remaining components are zero. Then $z\in Pl(k,n)$ and \eqref{eq_pl_vect} reduced to $t^{b_{r_j}+b_{s_j}}=t^{b_{r_i}+b_{s_i}}$, for all $t\in \CC^{*}$.  This implies $b_{r_i}+b_{s_i}=b_{r_j}+b_{s_j}$. Thus both \eqref{eq_plu_coord} and \eqref{eq_pl_vect} satisfies together for any $t\in \CC^*$ and for every $z\in Pl(k,n)$ iff $b_{r_j}+b_{s_j}=b_{r_i}+b_{s_i}$ is constant for all $i,j$ in $\{0,1,\dots,k\}$. This motivates us to introduce the Pl\"{u}cker weight vector. 
\end{remark}

 In the next definition and throughout the paper, we assume $k$ and $n$ are two positive integers with $k<n$ and $m+1={n \choose k}$.
 \begin{definition}\label{def_plu_vec}
 A weight vector ${\bf{b}}=(b_0, b_1,\dots, b_m)\in (\mathbb{Z}_{\geq 1})^{m+1}$ is called Pl\"{u}cker weight vector if for every Pl\"{u}cker relation $\mathfrak{p}$ 
$$b_{r_j^\mathfrak{p}}+b_{s_j^\mathfrak{p}}=b_{r_i^\mathfrak{p}}+b_{s_i^\mathfrak{p}} \text{ for all } i,j \text{ in }\{0,1,\dots,k\},
$$ where the $(k+1)$ many pairs $\{(r_0^\mathfrak{p},s_0^\mathfrak{p}), (r_1^\mathfrak{p},s_1^\mathfrak{p}), \dots, (r_k^\mathfrak{p},s_k^\mathfrak{p})\}$ are uniquely defined for every Pl\"{u}cker relation $\mathfrak{p}$. 
 \end{definition}

\begin{remark}\label{rem_plu_wgt_vec}
Let ${\bf b}=(b_0,b_1,\dots,b_m)$ be a Pl\"{u}cker weight vector. If  $\lambda^i\cup \lambda^j=\lambda^r\cup \lambda^q$ for any 4 Schubert symbols $\lambda^i$, $\lambda^j$, $\lambda^q$, $\lambda^r$ then $b_i+b_j=b_r+b_q$.  
\end{remark}

\begin{example}\label{ex_pl_wgt_vec}
	For $k=2$ and $n=4$, we have $m={4\choose2}-1=5$. In this case, we have $6$ Schubert symbols which are $$\lambda^0=(1,2),\lambda^1=(1,3),\lambda^2=(1,4),\lambda^3=(2,3),\lambda^4=(2,4)~ \text{and}~ \lambda^5=(3,4)$$ in the total ordering as in \eqref{eq_total_ord}.
	Consider $z=(z_0,z_1,\dots,z_5)\in \CC^6\setminus \{0\}$. Let  $$i_1=1,~ \ell_0=2<\ell_1=3<\ell_2=4.$$ Then $z\in Pl(2,4)$ iff $z_0z_5-z_1z_4+z_2z_3=0$. A weight vector ${\bf{b}}=(b_0,b_1,\dots,b_5)\in (\ZZ_{\geq 1})^6$ is a Pl\"{u}cker weight vector if $b_0+b_5=b_1+b_4=b_2+b_3$.  
\end{example}

Using Remark \ref{rmk_pl_rln}, it follows that $Pl(k,n)$ is an invariant subset of $\CC^{m+1}\setminus\{{\bf 0}\}$ under the ${\bf b}$-action of $\CC^*$ iff ${\bf b}$ is a Pl\"{u}cker weight vector. This paves the way for the definition of the following:

\begin{definition}\label{def_wgt_gsm}
    Let ${\bf{b}}=(b_0, b_1,\dots, b_m)\in (\mathbb{Z}_{\geq 1})^{m+1}$ be a Pl\"{u}cker weight vector. We define a topological space $\G_{\bf{b}}(k,n)$ as the orbit space of the ${\bf b}$-action of $\CC^*$ on $Pl(k,n)$ as in \eqref{eq_weigh_act}. Thus  $$\G_{\bf{b}}(k,n):=\frac{Pl(k,n)}{{\bf b}\text{-action}}.$$
	Consider the quotient map  $$\pi_{\bf{b}}: Pl(k,n)\to \G_{\bf{b}}(k,n)$$ defined by $\pi_{\bf{b}}(z)=[z]$, where $[z]$ is the orbit of $z$ with respect to the ${\bf b}$-action. The topology on $\G_{\bf{b}}(k,n)$ is given by the quotient topology induced by the map $\pi_{\bf{b}}$. 
\end{definition}

\begin{remark}
  If ${\bf{b}}=(1,1,\dots,1)$ then ${\bf{b}}$ is a Pl\"{u}cker weight vector. In this case the space $\G_{\bf{b}}(k,n)$ is the Grassmann manifold $\G(k,n)$. We denote the corresponding quotient map by $$\pi: Pl(k,n)\to \G(k,n).$$
\end{remark}

Now we show that the space $\G_{\bf b}(k,n)$ in Definition \ref{def_wgt_gsm} is a generalization of weighted Grassmannians  $\WG(k,n)$ in \cite[Definition 2.2]{AbMa1}. Let $W:=(w_1,w_2, \dots ,w_n) \in (\mathbb{Z}_{\geq 0})^{n}$ and $a \in\mathbb{Z}_{\geq 1}$. For $k<n$, recall $Pl(k,n)$ from \eqref{eq_pl_kn}. The algebraic torus $(\CC^*)^{n+1}$ acts on $Pl(k,n)$ by the following
\begin{equation}\label{eq_c_action}
	(t_1, t_2, \dots, t_n, t) \sum_{i=0}^m z_{i}e_{\lambda^i}=\sum_{i=0}^m t{\cdot}t_{\lambda^i} z_{i} e_{\lambda^i}.
\end{equation}	

Consider a subgroup $WD$ of $(\CC^*)^{n+1}$ defined by \begin{equation}\label{eq_wd}
	   	WD:=\{(t^{w_1},t^{w_2},\dots,t^{w_n},t^a):t\in \CC^*\}.  
\end{equation}
Then there is a restricted action of $WD$ on $Pl(k,n)$. Abe and Matsumura denoted the quotient $\frac{Pl(k,n)}{WD}$  by $\WG(k,n)$ and call it a weighted Grassmannian in \cite{AbMa1}. For all $i\in\{0,1,\dots,m\}$, let  
\begin{equation}\label{eq_wli}
	b_i=w_{\lambda^i} :=a+\sum_{j=1}^{k} w_{\lambda_j^i},
\end{equation}
where $\lambda^i = (\lambda_1^i,\lambda_2^i,\dots,\lambda_k^i)$ is the $i$-th Schubert symbol given in \eqref{eq_total_ord}. Then $b_i\geq 1$ for all $i\in\{0,1,\dots,m\}$. Consider the weight vector ${\bf{b}}=(b_0,b_1,\dots,b_m)\in (\ZZ_{\geq 1})^{m+1}$. From the Definition \ref{def_plu_vec}, it follows that ${\bf{b}}$ is a Pl\"{u}cker weight vector. Note that action of $WD$  on $Pl(k,n)$ is same as the ${\bf{b}}$-action of $\CC^*$ on $Pl(k,n)$. Thus $\G_{\bf{b}}(k,n)$ is same as the weighted Grassmannian $\WG(k,n)$, where the Pl\"{u}cker weight vector ${\bf{b}}=(b_0,\dots,b_m)$ is defined using \eqref{eq_wli}.

\begin{remark}\label{rem_2_4_case}
  Let $n=4$ and $k=2$. Then $m={4\choose 2}-1=5$. Recall the $6$ Schubert symbols from Example \ref{ex_pl_wgt_vec}.  Consider a weight vector ${\bf{b}}=(5,1,4,3,6,2)\in (\ZZ_{\geq 1})^6$. Then ${\bf{b}}$ is a Pl\"{u}cker weight vector follows from Example \ref{ex_pl_wgt_vec}. If there exist $W=(w_1,w_2,w_3,w_4)\in (\ZZ_{\geq 0})^4$ and $a\in \ZZ_{\geq 1}$ such that $b_i=w_{\lambda^i}$ then $b_1=w_1+w_3+a=1$. Now $W\in (\ZZ_{\geq 0})^4$ implies that $b_1\geq a$. The only choice for $a$ is $1$. Then we have $w_1+w_3=0$. Also $w_1-w_3=b_0-b_3=2$. Which gives $w_1=1$, $w_3=-1$, $w_2=3$ and $w_4=2$. Thus $W=(1,-1,3,2)\notin (\ZZ_{\geq 0})^4$. 
  
  If we choose another Pl\"{u}cker weight vector ${\bf{b}}=(5,1,4,4,7,3)$, then by a similar calculation as above, we get $W=(1/2,7/2,-1/2,5/2)$ and $a=1$. Moreover, ${\bf b}$ also corresponds to $W=(0,3,-1,2)$ and $a=2$.
\end{remark}

In the following Proposition, we prove that for a Pl\"{u}cker weight vector ${\bf b}=(b_0,\dots,b_m)$, there exist solutions of $W$ and $a$ in a broader class.

\begin{proposition}\label{prop_rel_two_wt}
Let ${\bf{b}}=(b_0, b_1,\dots, b_m)\in (\mathbb{Z}_{\geq 1})^{m+1}$ be a Pl\"{u}cker weight vector for $k<n$. Then there exist $a\in \{1,2,\dots,k\}$ and $W=(w_1,w_2,\dots,w_n) \in \mathbb{Z}^{n}$ such that $b_i=w_{\lambda^i}$ for all $i=0,1,\dots,m$, where $w_{\lambda^i}$ is defined in \eqref{eq_wli}.
\end{proposition}
\begin{proof}
Consider the linear equations defined by $w_{\lambda^i}=b_i$ for all $i=0,1,\dots,m$ with variables $w_1,w_2,\dots,w_n \text { and }a$. Let us fix the variable $w_1$. For every $j$ satisfying $2\leq j \leq n$ we can choose two Schubert symbols $\lambda^{r_j}$ and $\lambda^{s_j}$ such that $(1,j)\in \text{inv}(\lambda^{r_j})$ and $(1,j)\lambda^{r_j}=\lambda^{s_j}$.
Then we have linear equations of the form 
\begin{equation}\label{eq_int b_i}
w_1-w_{j}=b_{r_j}-b_{s_j};\text{ for } 2\leq j\leq n.	
\end{equation}
Note that if there exists two Schubert symbols $\lambda^{q_j}$ and $\lambda^{t_j}$ such that $(1,j)\in \text{inv}(\lambda^{q_j})$ and $(1,j)\lambda^{q_j}=\lambda^{t_j}$ for then $\lambda^{r_j}\cup \lambda^{t_j}=\lambda^{q_j}\cup \lambda^{s_j}$. Since ${\bf b}$ is a Pl\"{u}cker weight vector, using Remark \ref{rem_plu_wgt_vec}, $b_{r_j}+b_{t_j}=b_{q_j}+b_{s_j}$. So,  $b_{r_j}-b_{s_j}=b_{q_j}-b_{t_j}$. Thus \eqref{eq_int b_i} is well defined if ${\bf b}$ is a Pl\"{u}cker weight vector.  
   
 Now fix a Schubert symbol $\lambda^i=(2,3,\dots,k+1)$. We add the $k$ equations of \eqref{eq_int b_i} for $2\leq j\leq k+1$ with the equation $w_{\lambda^i}=b_i$ and we get
  $$kw_1+a=b_i+\sum_{j=2}^{k+1}(b_{r_j}-b_{s_j}).$$ 
  Note that the right side of the above equation is an integer. This gives a solution for $a \in \{1,2,\dots,k\}$ which satisfies
  $$b_i+\sum_{j=2}^{k+1}(b_{r_j}-b_{s_j}) \equiv a ~(\text{mod } k).$$
Thus $kw_1 \equiv 0 ~(\text{mod } k)$. 
Then, we get an integer solution for $w_1$. Consequently, we get integer solutions for $w_j$ for all $2\leq j\leq n$ using \eqref{eq_int b_i}. Thus all the $w_j$ are integers.   
\end{proof}

\begin{remark}
The Pl\"{u}cker weight vector corresponding to $W=(w_1,\dots,w_n)\in \ZZ^n$ and $a\in \ZZ$ is same as the  Pl\"{u}cker weight vector corresponding to $W_\alpha=(w_1-\alpha,\dots,w_n-\alpha)\in \ZZ^n$ and $a_\alpha=a+k\alpha$ for any $\alpha \in \ZZ$.
\end{remark}
Thus, Definition \ref{def_wgt_gsm} broadens the class of weighted Grassmannian in \cite{AbMa1} as well as \cite{BS} in terms of Pl\"{u}cker weight vector. The authors in \cite{AbMa1} discussed the orbifold structure of $\WG(k,n)$ for $W=(w_1,\dots,w_n)\in (\ZZ_{\geq 0})^n$ and $a\in \ZZ_{\geq 1}$. Using the argument of \cite[Subsection 2.2]{AbMa1} and Proposition \ref{prop_rel_two_wt}, the quotient space $\G_{\bf{b}}(k,n)$ has an orbifold structure for a Pl\"{u}cker weight vector ${\bf{b}}$. We call the space $\G_{\bf{b}}(k,n)$ a weighted Grassmann orbifold associated to the Pl\"{u}cker weight vector ${\bf{b}}$.

\subsection{$q$-CW complex structures of weighted Grassmann orbifolds}\label{subsec_q_cell_stu}
  
In this subsection, we discuss a $q$-CW complex structure of $\G_{\bf{b}}(k,n)$. This $q$-CW complex structure induces a building sequence on $\G_{\bf b}(k,n)$ and we explore the cofibrations that appear at each stage of the building sequence. A CW complex structure of $\G(k,n)$ given by 
\begin{equation*}
  \G(k,n)=\bigsqcup_{i=0}^m {E}(\lambda^i), \mbox{ where } 
\end{equation*}
\begin{equation}\label{def_e_lam}
    {E}(\lambda^i)=\{[(z_0,z_1,\dots,z_m)]\in \G(k,n)~|~ z_{i}=1 ,z_{j}= 0 {\mbox{ for }} j<i\}.
\end{equation}

 Let $\lambda^i=(\lambda^i_1,\lambda^i_2,\dots,\lambda^i_k)$ denote the $i$-th Schubert symbol. Then every element of $E(\lambda^i)$ has a presentation $[A]=[a_1\wedge a_2\wedge\dots \wedge a_k]$, where 
 \begin{equation}\label{eq_row_mat}
a_s:=(a_{s1},a_{s2},\dots,a_{sn})\in \CC^n, 
\end{equation}
such that ${a}_{s\lambda^i_s}=1$ and ${a}_{sj}=0$ for $j<\lambda^i_s$ and $j\in \{\lambda^i_{s+1},\dots,\lambda^i_k\}$,  see \cite[Section 1.1]{ike}.

Then $E(\lambda^i)$ is an open cell of $\G(k,n)$ of codimension $2d(\lambda^i)$, where
$$d(\lambda^i):=(\lambda^{i}_1-1)+(\lambda^{i}_2-2)+ \cdots + (\lambda^{i}_k-k).$$

Consider $\widetilde{E}(\lambda^i):=\pi^{-1}(E(\lambda^i))\subset Pl(k,n)$. Then $Pl(k,n)=\bigsqcup_{i=0}^m\widetilde{E}(\lambda^i)$. Moreover, $$\widetilde{E}(\lambda^i)=\{[(z_0,z_1,\dots,z_m)]\in Pl(k,n)~|~ z_{i}\neq0 ,z_{j}= 0 {\mbox{ for }} j<i\}.$$ Then $\widetilde{E}(\lambda^i)$ is homeomorphic to ${E}(\lambda^i)\times \CC^{*}$ via the map
$$f:\widetilde{E}(\lambda^i)\to {E}(\lambda^i)\times \CC^{*} \text{ defined by } f(z)=([z],z_i).$$   The inverse map of $f$ is defined by
$${E}(\lambda^i)\times \CC^{*}\to \widetilde{E}(\lambda^i); ~~([z],s) \to s(z_i)^{-1}z.$$

 The $\bf b$-action of $\CC^*$ on $Pl(k,n)$ restricts a $\bf b$-action of $\CC^*$ on $\widetilde{E}(\lambda^i)$. Define a $\bf b$-action of $\CC^*$ on $E(\lambda^i)$ by 
\begin{align}\label{eq_t_act}
    t[(z_0,\dots,z_m)]=[(t^{b_0}z_0,\dots,t^{b_m}z_m)] \text{ for } t\in \CC^*.
\end{align}

 This gives rise to a ${\bf{b}}$-action of $\CC^*$ on ${E}(\lambda^i)\times \CC^{*}$ defined by $$t([z],s)=([tz], t^{b_i}s).$$

Then $f$ is $\CC^*$-equivariant with respect to the above ${\bf{b}}$-actions. Let $G(b_i)$ be the finite subgroup of $\CC^*$ defined by 
$$G(b_i)=\{t\in \CC^*~|~t^{b_i}=1\}.$$
Thus using \cite[Lemma 2.6]{BS}

$$ \pi_{\bf{b}}(\widetilde{E}(\lambda^i)) \cong \frac{{E}(\lambda^i)\times \CC^{*}}{{\bf b}\text{-action}} \cong \frac{{E}(\lambda^i)}{G(b_i)}.$$

Therefore a $q$-cell decomposition of $\G_{\bf b}(k,n)$ is given by $\G_{\bf{b}}(k,n)=\sqcup_{i=0}^m\frac{E(\lambda^i)}{G(b_i)}.$

\subsection{Lens complex and induced cofibration}
Next, we recall lens complex from \cite{Ka}. Let $(c_0, c_1,\dots, c_m)\in (\mathbb{Z}_{\geq 1})^{m+1}$. For each $i\in \{0,1,\dots,m\}$ consider a unit sphere $$S^{2m-1}_i=\{(z_0,\dots,z_{i-1},0,z_i\dots,z_m)\in \CC^{m+1}\setminus \{{\bf 0}\}~|~\sum_{j\neq i;~j=0}^{m}|z_j|^2=1\}.$$  

\begin{definition}\cite{Ka}\label{def_twist_act}
 Consider an action of $G(c_i)$ on $S^{2m-1}_i$ defined by 
\begin{equation}\label{gcm act}
    t(z_0,z_1,\dots, z_{i-1},0,z_{i+1},\dots,z_{m})=(t^{c_0}z_0,t^{c_1}z_1,\dots,t^{c_{i-1}}z_{i-1},0,t^{c_{i+1}}z_{i+1},\dots,t^{c_{m}}z_{m}),
\end{equation}
	where $t\in G(c_i)$. The quotient space $$L'(c_i;(c_0,c_1,\dots,c_{i-1},c_{i+1},\dots,c_m)):= \dfrac{S^{2m-1}_i}{G(c_i)}$$ is called the lens complex.
	If $\{c_j:0\leq j\leq m,j\neq i\}$ are prime to $c_i$ then $L'(c_i;(c_0,\dots,c_{i-1},$ $c_{i+1},\dots,c_m))$ is known as usual generalized lens space. 
\end{definition}

Next, we show that the lens complex arises naturally as a cofiber in the building sequence induced from the $q$-CW complex structure of $\G_{\bf b}(k,n)$.
For each $i \in \{0,1,\dots, m\}$, define $X^{m-i}:=\bigsqcup_{j\geq i}\frac{E(\lambda^j)}{G(b_j)}.$ 
 The building sequence corresponding to the $q$-CW complex structure of $\G_{\bf b}(k,n)$ is
\begin{equation}\label{filtration}
	\{pt\}=X^0 \subset X^1 \subset X^2 \subset \dots \subset X^m=\G_{\bf b}(k,n).
\end{equation} 

For every $i\in \{0,1,\dots,m\}$ consider a subset of Schubert symbol 
\begin{equation}\label{eq_R_lambda}
I(\lambda^i)=\{\lambda^j: \lambda^j=(\lambda^i_s,r)\lambda^i \text{ for } (\lambda^i_s,r)\in {\rm{inv}}(\lambda^i) \text{ and } s=1,2,\dots,k\}.
\end{equation}

Note that $|I(\lambda^i)|=k(n-k)-d(\lambda^i)$ for all $i=0,1,\dots,m$. Let $d'(\lambda^i):=k(n-k)-d(\lambda^i)$. Consider $ 
    \CC^{d'(\lambda^i)}\subset \CC^{m+1}$ by \begin{equation}\label{eq_c_d_i}
    \CC^{d'(\lambda^i)}=\{(z_0,z_1,\dots,z_m)~|~z_j=0 \text{ if } \lambda^j\notin I(\lambda^i)\}.
\end{equation}

Then, ${E}(\lambda^i)$ is homeomorphic to $\CC^{d'(\lambda^i)}$ by the following map
$$g\colon {E}(\lambda^i)\to \CC^{d'(\lambda^i)}.$$

 The $j$-th entry of
$g([a_1\wedge\dots\wedge a_k])$ is equal to $a_{sr}$ (the $r$-th entry of $a_s$) if $\lambda^j=(\lambda^i_s,r)\lambda^i$ for some $1\leq s\leq k$, $1\leq r\leq n$ and $r\notin\{\lambda^i_1,\dots,\lambda^i_k\}$. 
This defines both $g$ and $g^{-1}$ completely and both $g$ and $g^{-1}$ are continuous.

 Consider the $G(b_i)$-action on $\CC^{d'(\lambda^i)}$ by the restriction of the ${\bf{b}}$-action of $\CC^*$ on $\CC^{m+1}$ as in \eqref{eq_weigh_act}. Recall the $G(b_i)$ action on $E(\lambda^i)$ as in \eqref{eq_t_act}. Then $g$ is a $G(b_i)$-equivariant homeomorphism. Moreover,
$$ \frac{E(\lambda^i)}{G(b_i)} \cong \frac{\CC^{d'(\lambda^i)}}{G(b_i)}.$$

Next, we define weight vectors $ {\bf{b}}^{(i)}$ for every $i=0,1,\dots,m$ induced from a  Pl\"{u}cker weight vector ${\bf b}=(b_0,b_1,\dots,b_m)$. The weight vector $ {\bf{b}}^{(i)}$ is defined by substituting $b_j=0$ in $(b_0,b_1,\dots,b_m)$ if $\lambda^j\notin I(\lambda^i)$. Moreover, for every $i=0,1,\dots,m$, we define a lens complex $L'(b_i;{\bf{b}}^{(i)})\subset L'(b_i;(b_0,\dots,b_{i-1},b_{i+1},\ldots,b_{m}))$ by the following:
$$L'(b_i;{\bf{b}}^{(i)}):=\{[(z_0,z_1,\dots,z_{m})]\in L'(b_i;(b_0,\dots,b_{i-1},b_{i+1},\ldots,b_{m}))~|~z_j=0 \text{ if } \lambda^j \notin I(\lambda^i) \}$$
\begin{proposition}\label{prop_cofibre}
	$L'(b_i;{\bf{b}}^{(i)}) \hookrightarrow X^{i-1} \to X^i$ is a cofibration for every $i=1,\dots,m$.
\end{proposition}
\begin{proof}
We consider $\CC^{d'(\lambda^{i})}$ from \eqref{eq_c_d_i} and the unit circle on $\CC^{d'(\lambda^{i})}$ by $$S^{2d'(\lambda^{i})-1}=\{(z_0,\dots,z_m)\in \CC^{d'(\lambda^{i})}~|~ \sum_{i=0}^m|z_i|^2=1\}. \text{ Then }$$  $$L'(b_i;{\bf{b}}^{(i)})= \dfrac{S^{2d'(\lambda^{i})-1}}{G(b_i)}.$$ Using the $q$-CW complex decomposition $X^i=X^{i-1}\sqcup \frac{\CC^{d'(\lambda^{i})}}{G(b_i)}$. Now the result follows from the fact that $\frac{\CC^{d'(\lambda^{i})}}{G(b_i)}$ is an unbounded cone over $L'(b_i;{\bf{b}}^{(i)})$.
\end{proof}

\section{Classification of weighted Grassmann orbifolds upto certain homeomorphism}\label{sec_clss_prob}

In this section, we introduce the Pl\"{u}cker permutation, a permutation on Pl\"{u}cker coordinates up to some sign. Then, we prove that a Pl\"{u}cker permutation induces an equivariant homeomorphism between two weighted Grassmann orbifolds. We also show that two weighted Grassmann orbifolds are homeomorphic if their Pl\"{u}cker weight vectors are differed by a scalar multiplication. We discuss certain equivariant rigidity of weighted Grassmann orbifolds in terms of {\Pl} weight vector. Moreover, we provide a conjecture that classify weighted Grassmann orbifolds up to homeomorphism.

Let $\sigma$ be a permutation on $\{0,1,2,\dots,m\}$ and $z=(z_0,\dots,z_m)\in \CC^{m+1}\setminus \{{\bf 0}\}$. Define $\sigma z:=(z_{\sigma(0)},z_{\sigma(1)},\dots,z_{\sigma(m)})$. Consider a sign vector $t=(t_0,\dots,t_m)\in \{1,-1\}^{m+1}$, i.e., $t_i\in \{1,-1\}$ for all $i=0,1,\dots,m$. Define $t{\cdot}\sigma z:=(t_0z_{\sigma(0)},t_1z_{\sigma(1)},\dots,t_mz_{\sigma(m)})$.

\begin{definition}
A permutation $\sigma$ on $\{0,1,\dots,m\}$ is said to be Pl\"{u}cker permutation if there exist a sign vector $t_\sigma\in \{1,-1\}^{m+1}$ such that $t_{\sigma}{\cdot}\sigma z\in Pl(k,n)$ for every $z\in Pl(k,n)$. 
\end{definition}

Let $\sigma$ and $\eta$ be two Pl\"{u}cker permutations on $\{0,1,2,\dots,m\}$.  Then $\sigma\circ \eta$ is also a Pl\"{u}cker permutation on $\{0,1,2,\dots,m\}$. Note that, $\sigma^{-1}=\sigma^i$ for some positive integer $i$. Thus $\sigma^{-1}$  is also a Pl\"{u}cker permutation on $\{0,1,2,\dots,m\}$. Hence the Pl\"{u}cker permutations on $\{0,1,\dots,m\}$ form a subgroup of the permutation group on $\{0,1,\dots,m\}$. 

\begin{remark}
 The permutation group $S_n$ can be viewed as a subgroup of the Pl\"{u}cker permutations on $\{0,1,\dots,m\}$ if $m+1={n\choose k}$. To verify this, consider a Schubert symbol $\lambda^i =(\lambda_1^i, \ldots, \lambda_k^i)$ for $i\in \{0,1, \ldots, m\}$ and a permutation $\phi \in S_n$. Define $\phi\lambda^i :=(\phi(\lambda_{i_1}^i), \ldots, \phi(\lambda_{i_k}^i))$, where $\{i_1,\dots,i_k\}=\{1,\dots,k\}$ such that $1\leq \phi(\lambda_{i_1}^i)<\cdots <\phi(\lambda_{i_k}^i)\leq n$. Thus $\phi$ gives rise to a permutation $\sigma$ on $\{0,1,\dots,m\}$ defined by $\phi \lambda^i=\lambda^{\sigma(i)}$. Moreover, $\sigma$ is a Pl\"{u}cker permutation on $\{0,1,\dots,m\}$ follows from \eqref{eq_plu_coord}. 

There are many Pl\"{u}cker permutations on $\{0,1,\dots,m\}$ those do no arise from $S_n$. For example Consider $d=2$, $n=4$, $m+1=6$. Recall the total order on Schubert symbol as in Example \ref{ex_pl_wgt_vec}. Define a permutation $\sigma\in \{0,1,2,3,4,5\}$ by $\sigma(0)=5, \sigma(5)=0$ and $\sigma(i)=i$ for the remaining $i$. Then it is a Pl\"{u}cker permutation corresponding to the sign vector $t_\sigma=(1,1,1,1,1,1)$. Note that there does not exist $\phi\in S_4$ such that $\phi \lambda^i=\lambda^{\sigma(i)}$. 
\end{remark}

\begin{lemma}
 Let $\sigma$ be a Pl\"{u}cker permutation. Then $\sigma {\bf{b}}$ is a Pl\"{u}cker weight vector for every Pl\"{u}cker weight vector ${\bf{b}}$.	
\end{lemma}
\begin{proof}
Let ${\bf{b}}=(b_0,b_1,\dots,b_m)$ be a Pl\"{u}cker weight vector. Then $(t^{b_0}z_0,\dots,t^{b_m}z_m)\in Pl(k,n)$ for every  $z=(z_0,\dots,z_m)\in Pl(k,n)$ and $t\in \CC^*$. Let $\sigma$ is a Pl\"{u}cker permutation on $\{0,1,\dots,m\}$. Then there exists $t_{\sigma}=(t_0,\dots,t_m)\in \{1,-1\}^{m+1}$ such that for every  $z=(z_0,\dots,z_m)\in Pl(k,n)$ we have $t_\sigma{\cdot}\sigma z= (t_0z_{\sigma(0
)},\dots,t_mz_{\sigma(m)})\in Pl(k,n)$ as well as $t_{\sigma}{\cdot} \sigma(t^{b_0}z_0,\dots,t^{b_m}z_m)=(t_0t^{b_{\sigma(0)}}z_{\sigma(0)},\dots,t_mt^{b_{\sigma(m)}}z_{\sigma(m)})\in Pl(k,n)$. Thus using Remark \ref{rmk_pl_rln}, it follows that $\sigma {\bf{b}}=(b_{\sigma(0)},\dots,b_{\sigma(m)})$ is a Pl\"{u}cker weight vector. 
\end{proof}

Pl\"{u}cker permutations play an important role in characterizing the weighted Grassmann orbifold up to a weakly equivariant homeomorphism. Consider the $(\CC^*)^n$-action on $\G_{\bf{b}}(k,n)$ defined by
\begin{equation}\label{eq_c_act}
	(t_1, t_2, \dots, t_n) \sum_{i=0}^m [a_{i}e_{\lambda^i}]=\sum_{i=0}^m [t_{\lambda^i} a_{i} e_{\lambda^i}],
\end{equation} 
where $ t_{\lambda} :=t_{\lambda_1} \cdots t_{\lambda_{k}}$, for the Schubert symbol $\lambda=(\lambda_1, \ldots, \lambda_k)$, for $k < n$.

\begin{lemma}\label{thm_per_plu_vec}
	If $\sigma$ is a Pl\"{u}cker permutation then $\G_{\bf{b}}(k,n)$ is weakly equivariantly homeomorphic to $\G_{\sigma {\bf{b}}}(k,n)$. 
\end{lemma}
\begin{proof}
A Pl\"{u}cker permutation $\sigma$ induces the natural weakly equivariant homeomorphism $\bar{f}_{\sigma} \colon 	Pl(k,n) \to Pl(k,n) $ defined by $\bar{f}_{\sigma}(z)=t_\sigma{\cdot}\sigma z$. Then $\bar{f}_{\sigma}$ induces a homeomorphism $f_\sigma$ between  $\G_{\bf{b}}(k,n)$ and $\G_{\sigma {\bf{b}}}(k,n)$ such that the following diagram commutes.  
	\[ \begin{tikzcd}
		Pl(k,n) \arrow{r}{\bar{f}_\sigma} \arrow{d}{\pi_{\bf{b}}} & Pl(k,n) \arrow{d}{\pi_{\sigma {\bf{b}}}} \\%
		\G_{\bf{b}}(k,n) \arrow{r}{f_{\sigma}} & \G_{\sigma {\bf{b}}}(k,n).
	\end{tikzcd}
	\]
 The homeomorphism $f_{\sigma}$ becomes weakly equivariant with respect to the permuted $(\CC^*)^n$-action on $\G_{\sigma {\bf b}}(k,n)$ defined by
 \begin{equation*}
	(t_1, t_2, \dots, t_n) \sum_{i=0}^m [a_{i}e_{\lambda^i}]=\sum_{i=0}^m [t_{\lambda^{\sigma(i)}} a_{i} e_{\lambda^i}]
\end{equation*} 
\end{proof}

\begin{remark}\label{rem_per_cell}
 If $\sigma$ is a {\Pl} permutation on $\{0,1,\dots,m\}$ then $\Big\{\frac{\CC^{d'(\lambda^i)})}{G(b_{\sigma(i)})}\Big\}_{i=0}^m$ gives a $q$-CW complex structure of $\G_{\sigma{\bf b}}(k,n)$ following Subsection \ref{subsec_q_cell_stu}. 
  As a consequence of Lemma \ref{thm_per_plu_vec}, $\Big\{\frac{\CC^{d'(\lambda^i)})}{G(b_{\sigma(i)})}\Big\}_{i=0}^m$ gives another $q$-CW complex structure of $\G_{\bf{b}}(k,n)$.
\end{remark}

\begin{example}
    Consider a weighted Grassmann orbifold $\G_{\bf b}(2,4)$ for the  Pl\"{u}cker weight vector ${\bf{b}}=(3,2,1,4,3,2)$ which corresponds to $W=(0,2,1,0)$ and $a=1$. Consider the Pl\"{u}cker permutation $\sigma$ on $\{0,1,2,3,4,5\}$ defined by $\sigma(0)=5, \sigma(5)=0$ and $\sigma(i)=i$ for the remaining $i$. Then $\sigma$ is a {\Pl} permutation corresponding to $t_\sigma=(1,1,1,1,1,1)$. Now $\sigma {\bf{b}}=(2,2,1,4,3,3)$ which corresponds to $W'=(-1,1,1,0)$ and $a'=2$. Thus both the weighted Grassmann orbifolds $\WG(2,4)$ corresponding to $(W,a)$ and $(W',a')$ are homeomorphic $\G_{\bf b}(2,4)$ described above.  
\end{example}

Using Definition \ref{def_plu_vec}, it follows that if ${\bf b}=(b_0,b_1,\dots,b_m)$ is a Pl\"{u}cker weight vector then $r{\bf b}:=(rb_0,rb_1,\dots,rb_m)$ is also a Pl\"{u}cker weight vector for any positive integer $r$.

\begin{lemma}\label{lemma_int_red}
     $\G_{\bf b}(k,n)$ is  equivariantly homeomorphic to $\G_{r{\bf b}}(k,n)$ for every positive integer $r$.
\end{lemma}
\begin{proof}
Consider the identity map ${\rm id}\colon Pl(k,n) \to Pl(k,n)$. 
The quotient map $$\pi_{r\bf b}\colon Pl(k,n) \rightarrow \G_{r\bf b}(k,n)$$ induces a map $f$ from $\G_{r\bf b}(k,n)$ to $\G_{{\bf b}}(k,n)$ for any positive integer $r$ such that the following diagram commutes.
\[\xymatrix{
	Pl(k,n) \ar[r]^{\rm{id}} \ar[d]_{\pi_{r{\bf b}}} & Pl(k,n)  \ar[d]^{\pi_{{\bf b}}}\\
	\G_{r \bf b}(k,n) \ar[r]^f_{\cong}  & \G_{\bf b}(k,n).}
\]
Let \(U\) be an open subset of \(\G_{\bf b}(k,n)\).  
Then \((\pi_{\bf b})^{-1}(U) = (\pi_{r{\bf b}})^{-1} \circ f^{-1}(U)\).  
Since \(\pi_{\bf b}\) is a quotient map, the set \((\pi_{\bf b})^{-1}(U)\) is open in 
\(Pl(k,n)\).  
Thus \(f^{-1}(U)\) is an open subset of \(\G_{r{\bf b}}(k,n)\), because 
\(\pi_{r{\bf b}}\) is also a quotient map.  Therefore, \(f\) is continuous.
By a similar argument we can define a map $g$ from $\G_{\bf b}(k,n)$ to $\G_{r\bf b}(k,n)$ such that $$f\circ g={\rm id}_{\G_{\bf b}(k,n)},\text{  and } g\circ f={\rm id}_{\G_{r\bf b}(k,n)}.$$
The map $g=f^{-1}$ is also continuous by the same argument.
 The $(\CC^*)^n$-action on $\G_{\bf b}(k,n)$ and $\G_{r{\bf b}}(k,n)$ are defined in \eqref{eq_c_act}. Hence $f$ becomes a $(\CC^*)^n$-equivariant homeomorphism.
\end{proof}

\begin{theorem}\label{thm_cl_wgt_gsm}
    Two weighted Grassmann orbifolds $\G_{\bf b}(k,n)$ and $\G_{\bf c}(k,n)$ are homeomorphic if the corresponding  Pl\"{u}cker weight vectors ${\bf b}$ and ${\bf c}$ differ by a Pl\"{u}cker permutation and a Scalar multiplication.
\end{theorem}
\begin{proof}
    The proof follows from Lemma \ref{lemma_int_red} together Lemma \ref{thm_per_plu_vec}.
\end{proof}

A weight vector $(c_0,c_1,\dots,c_m)\in (\ZZ_{\geq 1})^{m+1}$ is said to be a primitive weight vector if $\gcd\{c_0,c_1,\dots,c_m\}=1$. Moreover, a primitive weight vector $(c_0,c_1,\dots,c_m)$ is said to be normalized if, for any prime $p$, there exist weights $c_i$ and $c_j$ for $i\neq j$ such that $p$ does not divides both $c_i$ and $c_j$. If a primitive weight vector is not normalized, then we can always transform it into a normalized vector which is unique up to an order.

\begin{proposition}[\cite{BFNR}]
    Two weighted projective spaces $\WW P(b_0,\dots,b_m)$ and $\WW P(c_0,\dots,c_m)$ are homeomorphic if and only if $(b_0,\dots,b_m)$ and $(c_0,\dots,c_m)$ have same normalized form up to an order.
\end{proposition}

\begin{lemma}\label{lem_nor_vec}
    Let $k\geq 2$ and ${\bf b}=(b_0,\dots,b_m)\in (\ZZ_{\geq 1})^{m+1}$ be a primitive Pl\"{u}cker weight vector such that $m+1={n \choose k}$. Then $(b_0,\dots,b_m)$ is a normalized Pl\"{u}cker weight vector.
\end{lemma}
\begin{proof}
Let $p$ be a prime and ${\bf b}=(b_0,\dots,b_m)\in (\ZZ_{\geq 1})^{m+1}$ be a primitive Pl\"{u}cker weight vector.  Then there exist at least one $b_i$ such that $p$ does not divide $b_i$. Let $\lambda^i=(\lambda_1^i,\lambda_2^i,\lambda_3^i,\dots,\lambda_k^i)$. We consider $1\leq q,r\leq n$ such that none of $q$ and $r$ is in $\lambda^i$. Define $\lambda^j=(q,r,\lambda_3^i,\dots,\lambda_k^i)$, $\lambda^s=(q,\lambda_2^i,\lambda_3^i,\dots,\lambda_k^i)$, $\lambda^\ell=(\lambda_1^i,r,\lambda_3^i,\dots,\lambda_k^i)$. Then using Remark \ref{rem_plu_wgt_vec}, $b_i+b_j=b_s+b_\ell$. Since $p$ does not divide $b_i$, $p$ does not divide  at least one of $b_j,b_k,b_\ell$. This is true for all prime $p$. Hence the proof.
\end{proof}

Define $$Ad(\lambda^i):=\{\lambda^j: |\lambda^j\cap\lambda^i|=k-1\}.$$ Then cardinality of $Ad(\lambda^i)$ is $k(n-k)$. Corresponding to $Ad(\lambda^i)$ we define 
$${\bf b}_{(i)}:=\{{b_j}:\lambda^j\in Ad(\lambda^i)\} \text{ and }\overline{\bf b}_{(i)}:={\bf b}_{(i)}\cup \{b_i\}.$$

 For $i\in \{0,1,\dots,m\}$ define   
\begin{equation}\label{eq_open_nbd}
   U_i:=\{[z_0,\dots,z_m]\in \G_{\bf b}(k,n)~|~ z_i\neq 0\}. 
\end{equation}

From Section \ref{bld_seq_wgt_gmn_obd}, it follows that
$U_i\cong \CC^{k(n-k)}/G(b_i)$. Note that $U_i$ is an open cone over $L'(b_i; {\bf b}_{(i)})$.
Let $t:= k(n-k)$.
Using \cite[Theorem 2]{Ka} we have
	\begin{align}\label{eq_cohom_lens}
		H^q(L'(b_{i};{\bf{b}}_{(i)});\ZZ)=
		\begin{cases}
			\ZZ &\quad\text{for}~q=0,2t-1 \\
			\ZZ_{\mu} &\quad \mu={l_e{ \overline{\bf{b}}_{(i)}}}/{l_e{{\bf{b}}_{(i)}}} ~\text{if}~ q=2e ~\text{for}~1\leq e\leq t-1 \\    
			\{0\} &\quad\text{ otherwise}.    
		\end{cases}       
	\end{align}
For a prime $p$ let $r_i$ be the maximum non negative integer such that $p^{r_i}$ divides $b_i$ for all $i\in \{0,1,\dots,m\}$. We call $p^{r_i}$ as the $p$-th content in $b_i$. Let $p^{r_1},\dots,p^{r_t}$ be the  $p$-th contents of the elements in ${\bf b}_{(i)}$ such that $r_1\leq \dots,\leq r_t$. Then
\begin{equation}\label{eq_p-torsion}
  l_e{\bf{b}}_{(i)} = \prod_{p:~\text{prime}}\prod_{i=1}^e p^{r_{t-i+1}} ~\text{for}~1\leq e\leq t-1.  
\end{equation}

\begin{lemma}\label{lem_prim}
 Let $k\geq 2$ and ${\bf b}=(b_0,\dots,b_m)\in (\ZZ_{\geq 1})^{m+1}$ be a primitive Pl\"{u}cker weight vector such that $m+1={n \choose k}$. Then $\overline{\bf{b}}_{(i)}$'s are primitive vectors for all $i=0,1,\dots,m$.
\end{lemma}
\begin{proof}
Fix one $i\in \{0,1,\dots,m\}$. If $b_j\notin \overline{\bf b}_{(i)}$ then $|\lambda^j\cap \lambda^i|=k-u\leq k-2$. Without loss of generality assume that $\lambda^i=(\lambda^i_1,\dots,\lambda^i_u,\lambda^i_{u+1},\dots,\lambda^i_k)$ and $\lambda^j=(\lambda^j_1,\dots,\lambda^j_u,\lambda^i_{u+1},\dots,\lambda^i_k)$, where $\lambda^i_1\neq \lambda^j_1,$ $\dots,$ $ \lambda^i_u\neq \lambda^j_u$. Consider Schubert symbols $\lambda^{s_1}=(\lambda^j_1,\lambda^i_2\dots,\lambda^i_u,\lambda^i_{u+1},\dots,\lambda^i_k), $ $\dots, \lambda^{s_u}=(\lambda^i_1,\dots,\lambda^i_{u-1},\lambda^j_u,\lambda^i_{u+1},\dots,\lambda^i_k)$. Then $|\lambda^i\cap \lambda^{s_v}|=k-1$ and $b_{s_v}\in {\bf b}_{(i)}$ for all $v=1,2,\dots,u$. Note that $$\bigcup_{v=1}^u\lambda^{s_v}=\lambda^j\bigcup_{(u-1)\text{ times}}\lambda^i. $$ Since ${\bf b}$ is a Pl\"{u}cker weight vector using remark \ref{rem_plu_wgt_vec}, there exist $b_{s_1},b_{s_2},\dots, b_{s_u}\in {\bf b}_{(i)}$ such that $$b_j+(u-1)b_i=b_{s_1}+b_{s_2}\dots+b_{s_u}.$$ If $d$ divides $b_i$ and every element of ${\bf b}_{(i)}$ then $d$ also divides $b_j$. Hence, $\overline{\bf b}_{(i)}$ is not a primitive vector contradicts the fact that ${\bf b}$ is a primitive Pl\"{u}cker weight vector.  
\end{proof}

We recall the coordinate elements $[e_{\lambda^i}]\in U_i$, for $i=0,1,\dots,m$, where $e_{\lambda^i}$ is defined in Subsection \ref{subsec_pl_wgt_vec}.

\begin{lemma}\label{lemma_int_coh}
 Let $k\geq 2$ and ${\bf b}=(b_0,\dots,b_m)$ be a primitive Pl\"{u}cker weight vector. Then $$H^{2t-1}({\rm{Gr}}_{\bf b}(k,n),{\rm{Gr}}_{\bf b}(k,n)\setminus [e_{\lambda^i}]~;~\ZZ)=\ZZ_{b_i}.$$
\end{lemma}
\begin{proof}
We recall $U_i$ from \eqref{eq_open_nbd}. Using excision theorem $$H^{2t-1}(\G_{\bf b}(k,n),\G_{\bf b}(k,n)\setminus[e_{\lambda^i}]~;~\ZZ)\cong H^{2t-1}(U_i,U_i\setminus[e_{\lambda^i}]~;~\ZZ).$$ Now using the fact that $U_i$ is an open cone over $L'(b_i; {\bf b}_{(i)})$ $$H^{2t-1}(U_i,U_i\setminus[e_{\lambda^i}]~;~\ZZ)\cong H^{2t-2}(L(b_i;{\bf{b}}_{(i)});\ZZ).$$
Following the notation of \eqref{eq_cohom_lens}, we need to prove
\begin{equation}\label{eq_cls}
  \frac{l_{t-1}\overline{\bf{b}}_{(i)}}{l_{t-1}{\bf{b}}_{(i)}}=b_i. 
\end{equation}
Since ${\bf b}$ is primitive, we have $\overline{\bf b}_{(i)}$ is also primitive for all $i=0,1,\dots,m$, using Lemma \ref{lem_prim}. For a prime $p$, if $p$ does not divide $b_i$, then the $p$-th content of both sides of the \eqref{eq_cls} is 1. Now let $p$ divides $b_i$. Since $\overline{\bf b}_{(i)}$ is primitive there exists at least one element $b_j$ in ${\bf b}_{(i)}$ such that $p$ does not divide $b_j$.  Note that for every element $b_j$ in ${\bf b}_{(i)}$ there exist $b_{j'}, b_s, b_{s'}\in {\bf b}_{(i)}$ such that $$b_j+b_{j'}=b_s+b_{s'}.$$ Thus there exist at least two integers in ${\bf b}_{(i)}$ such that both integers are not divisible by $p$. Then $p$-th content of $b_i$=$p$-th content of $\frac{l_{t-1}\overline{\bf{b}}_{(i)}}{l_{t-1}{\bf{b}}_{(i)}}$. This is true for all prime $p$. This completes the proof.
\end{proof}

\begin{theorem}\label{thm_cohom_rig}
Let $k\geq 2$. If there exists a homeomorphism $f$ between two weighted Grassmann orbifolds ${\rm{Gr}}_{\bf b}(k,n)$ and ${\rm{Gr}}_{\bf c}(k,n)$ such that  $f$ maps every coordinate element $[e_{\lambda^i}]$ to some coordinate element $[e_{\lambda^j}]$, where $i,j\in\{0,1,\dots,m\}$, then ${\bf b}$ and ${\bf c}$ are same up to a scalar multiplication and a permutation.
\end{theorem}
\begin{proof}
  Without loss of generality, we can assume that both the Pl\"{u}cker weight vectors ${\bf b}$ and ${\bf c}$ are primitive. Let $f$ be the homeomorphism between $\G_{\bf b}(k,n)$ and $\G_{\bf c}(k,n)$ such that for all $i\in \{0,1,\dots,m\}$, $f([e_{\lambda^i}])=[e_{\lambda^j}]$ for some $j\in \{0,1,\dots,m\}$. Then it induces a isomorphism between $H^{2t-1}(\G_{\bf b}(k,n), \G_{\bf b}(k,n)\setminus[e_{\lambda^i}];\ZZ)$ and $H^{2t-1}(\G_{\bf c}(k,n), \G_{\bf c}(k,n)\setminus[e_{\lambda^j}];\ZZ)$. This implies $\ZZ_{b_i}$ is isomorphic to $\ZZ_{c_j}$. Thus $b_i=c_j$. Define a permutation $\sigma$ on $\{0,1,\dots,m\}$ by $\sigma(i)=j$. Then ${\bf b}=\sigma {\bf c}$ for some permutation $\sigma$ on $\{0,1,\dots,m\}$.      
\end{proof}

We recall the $(\CC^*)^n$-action on $\G_{\bf b}(k,n)$ from \eqref{eq_c_act} and note that the fixed point set of $\G_{\bf b}(k,n)$ with respect to the $(\CC^*)^n$-action is $\{[e_{\lambda^i}]\}_{i=0}^m$.

\begin{corollary}\label{cor_cls_eq_hom}
Let $k\geq 2$ and $G$ be a group acts on a weighted Grassmann orbifold $\G_{\bf b}(k,n)$ with the fixed point set $\{[e_{\lambda^i}]\}_{i=0}^m$. If ${\rm{Gr}}_{\bf b}(k,n)$ is weakly $G$-equivariantly homeomorphic to ${\rm{Gr}}_{\bf c}(k,n)$, then $\bf b$ and $\bf c$ are the same upto a scalar multiplication and a permutation.
\end{corollary}
\begin{proof}
    A weakly equivariant homeomorphism maps a fixed point to a fixed point. The coordinate elements $\{[e_{\lambda^i}]\}_{i=0}^m$ are the fixed point set in $\G_{\bf b}(k,n)$ with respect to the action of $G$. Hence, the proof follows from the Theorem \ref{thm_cohom_rig}.
\end{proof}

In the view of Lemma \ref{lem_nor_vec}, Lemma \ref{lem_prim} and Lemma \ref{lemma_int_coh} we propose the following conjecture.
\begin{conj}
Let $k\geq 2$. If two weighted Grassmann orbifolds ${\rm{Gr}}_{\bf b}(k,n)$ and ${\rm{Gr}}_{\bf c}(k,n)$ are homeomorphic then then ${\bf b}$ and ${\bf c}$ are same up to a scalar multiplication and a  Pl\"{u}cker permutation.
\end{conj}

\section{Torsion's in the integral cohomology of weighted Grassmann orbifold}\label{sec_tor_int_coh_wgt_gsm} 
In this section, we prove that for a prime $p$ the integral cohomologies of weighted Grassmann orbifolds have no $p$-torsion under some assumptions.  
We discuss some non-trivial classes of weighted Grassmann orbifolds whose integral cohomologies have no torsions. In particular, the integral cohomology of a divisive weighted Grassmann orbifold is torsion-free and concentrated in even degrees. Recall $I(\lambda^i)$ from \eqref{eq_R_lambda}.

\begin{theorem}\label{thm_no_p_tor}
	Let $p$ be a prime number. If for every $j\leq m-3$, the $p$-th content of $b_{j}$ divides $d'(\lambda^j)-1$ many elements of the set $\{ b_{\ell} : \lambda^\ell \in I(\lambda^j)\}$ then the integral cohomology of ${\rm{Gr}}_{\bf b}(k,n)$ has no $p$-torsion.  
\end{theorem}
\begin{proof}
Recall the building sequence $$ X^0\subset X^1\subset\dots\subset X^m=\G_{\bf b}(k,n),$$ induced from the $q$-CW complex structure as in \eqref{filtration}. We prove that the integral cohomology of $X^i$ has no torsion for all $0\leq i\leq m$ by using induction on $i$. Note that $X^0=\{pt\}, X_1=\WW P(b_{m-1},b_m)$ and $X^2=\WW P(b_{m-2},b_{m-1},b_m)$ are weighted projective spaces. Thus, the integral cohomology of $X^i$ has no torsion for all $0\leq i\leq 2$ by \cite[Theorem 1]{Ka}. Let us assume $H^*(X^{i-1};\ZZ)$ has no $p$-torsion. We recall the following cofibration from Proposition \ref{prop_cofibre}
$$L'(b_{i} ; {\bf{b}}^{(i)}) \hookrightarrow X^{i-1} \to X^i,$$
where ${\bf b}^{(i)}$ is defined prior to Proposition \ref{prop_cofibre}. This will induce the following exact sequence in cohomology
{\begin{equation}\label{eq_ex_seq_hom_wgt_gsm}
		\to \widetilde{H}^{q-1}(L'(b_{i} ; {\bf{b}}^{(i)});\ZZ)\to H^q( X^i;\ZZ) \to  H^q(X^{i-1};\ZZ) \to \widetilde{H}^{q}(L'(b_{i} ; {\bf{b}}^{(i)});\ZZ)\to.
	\end{equation}}
Using \cite[Theorem 2]{Ka} we have
\begin{align*}
H^q(L'(b_{i}; {\bf{b}}^{(i)});\ZZ)=
\begin{cases}
\ZZ &\quad\text{for}~q=0,2 d'(\lambda^i)-1  \\
\ZZ_{\mu} &\quad \mu={l_e{\overline{\bf{b}}^{(i)}}}/{l_e{{\bf{b}}^{(i)}}} ~\text{if}~ q=2e ~\text{for}~1\leq e\leq d'(\lambda^i)-1 \\    \{0\} &\quad\text{ otherwise},    
\end{cases}       
\end{align*}
 where ${\overline{\bf{b}}^{(i)}}:={\bf{b}}^{(i)}\cup \{b_i\}.$ 
For any prime $p$ the $p$-th content in $l_e{{\bf{b}}^{(i)}}$ is the product of $e$ many maximum $p$-th contents of the elements in set $\{b_{\ell}\colon \lambda^\ell\in I(\lambda^i)\}$ as in \eqref{eq_p-torsion}. Thus the $p$-th content of $\mu$ is 1, if the $p$-th content of $b_{i}$ divides $e$ many element of the set $ \{b_{\ell}:\lambda^\ell \in I(\lambda^i)\}$. Now the maximum value of $e$ can be $d'(\lambda^i)-1$. Thus by our assumption $H^*(L'(b_{i};{\bf{b}}^{(i)});\ZZ)$ has no $p$-torsion. Thus $H^*( X^i,\ZZ)$ has no $p$-torsion using \eqref{eq_ex_seq_hom_wgt_gsm} and induction hypothesis. Hence, by induction the integral cohomology of $X^m=\G_{\bf b}(k,n)$ has no $p$-torsion.
\end{proof}

\begin{corollary}
For a prime $p$, if there exists a Pl\"{u}cker permutation $\sigma$  on $\{0,\dots,m\}$ such that for all $j\leq m-3$, the $p$-th content of $b_{\sigma (j)}$ divides $d'(\lambda^j)-1$ many elements of the set $\{ b_{\sigma(\ell)} : \lambda^\ell \in I(\lambda^j)\}$ then the integral cohomology of ${\rm{Gr}}_{\bf b}(k,n)$ has no $p$-torsion.  
\end{corollary}

\begin{proof}
Let $\sigma$ be a  a Pl\"{u}cker permutation. Consider the building sequence of $\G_{\bf b}(k,n)$ induced from the $q$-CW complex structure following Remark \ref{rem_per_cell}. Now the proof followes using the same argument as in the Proof of Theeorem \ref{thm_no_p_tor}.
\end{proof}

\begin{remark}\label{rmk_no_p_tor}
If Theorem \ref{thm_no_p_tor} is true for every prime $p$ then $H^*(\G_{\bf{b}}(k,n),\ZZ)$ is torsion-free and contrated in even degrees.
\end{remark}

\begin{example}\label{no_tor_cohom}
Let $n=4,k=2$ and $m={4\choose 2}-1=5$. In this case, we have $6$ Schubert symbols given by
$$\lambda^0=(1,2)<\lambda^1=(1,3)<\lambda^2=(1,4)<\lambda^3=(2,3)<\lambda^4=(2,4)<\lambda^5=(3,4).$$ 
\begin{enumerate}
    \item Let $\alpha, \beta,\gamma$ be three positive integers. Consider a {\Pl} vector
  $${\bf{b}}=( \gamma,\gamma,\alpha \gamma, (\alpha\beta-\alpha+1)\gamma,\alpha\beta\gamma,\alpha\beta\gamma)\in (\ZZ_{\geq 1})^6.$$ Then $b_i$ divides $b_j$ for all $\lambda^j\succeq \lambda^i$ for $i\leq 2$. Thus the integral cohomology of the $\G_{\bf{b}}(2,4)$ has no torsion using Theorem \ref{thm_no_p_tor}.
  
  \item Let $\alpha,\beta,\gamma,s$ and $t$ be positive integers. Consider a weight vector of the form 
$${\bf{b}}=(\alpha,t\alpha,st\alpha,(1+s\beta)t\alpha,(1+\beta)st\alpha,\gamma)\in (\ZZ_{\geq 1})^6.$$
In addition, if $\alpha+\gamma=t\alpha(1+s+s\beta)$ then  ${\bf{b}}$ is a Pl\"{u}cker weight vector. 
The integral cohomology of the $\G_{\bf{b}}(2,4)$ has no torsion using Theorem \ref{thm_no_p_tor}. 
\end{enumerate}
\end{example}

\begin{example}
Let $n=5,k=2$ and $m={5\choose 2}-1=9$. In this case, we have $10$ Schubert symbols given by
\begin{align*}
   &\lambda^0=(1,2)<\lambda^1=(1,3)<\lambda^2=(1,4)<\lambda^3=(1,5)<\lambda^4=(2,3)<\\
   &\lambda^5=(2,4)<\lambda^6
   =(2,5)<\lambda^7=(3,4)<\lambda^8=(3,5)<\lambda^9=(4,5). 
\end{align*}
\begin{enumerate}
    \item Let $\alpha, \beta,\gamma$ be three positive integers. Consider  $${\bf{b}}=(\gamma,\alpha\gamma,\alpha\gamma,\alpha\beta\gamma,\gamma,\gamma, (\alpha\beta-\alpha+1)\gamma,  \alpha\gamma, \alpha\beta\gamma,\alpha\beta\gamma)\in (\ZZ_{\geq 1})^{10}.$$
Then the integral cohomology of the $\G_{\bf{b}}(2,5)$ has no torsion using Theorem \ref{thm_no_p_tor}. 

 \item Let $\alpha, \beta,\gamma,s$ ant $t$ be positive integers with $\alpha+\gamma=t\alpha(1+s+s\beta)$. Consider
$${\bf{b}}=(\alpha,\alpha,t\alpha,(1+\beta)st\alpha,\alpha,t\alpha,(1+\beta)st\alpha,t\alpha,(1+\beta)st\alpha,\gamma)\in (\ZZ_{\geq 1})^{10}.$$
Then the integral cohomology of the $\G_{\bf{b}}(2,5)$ has no torsion using Theorem \ref{thm_no_p_tor}.
\end{enumerate}
\end{example}

Next, we define divisive weighted Grassmann orbifold in terms of Pl\"{u}cker weight vector and Pl\"{u}cker permutation following \cite{BS}.

\begin{definition}
A weighted Grassmann orbifold $\G_{\bf b}(k,n)$ is called divisive if there exists a Pl\"{u}cker permutation $\sigma$ on $\{0,1,\dots,m\}$ such that $b_{\sigma(i)}$ divides $b_{\sigma(i-1)}$ for every $i=1,2,\dots,m$, where $m+1={n \choose k}$.
\end{definition}

\begin{example}
    Let $\alpha$ and $\beta$ be two positive integers. Consider a weight vector ${\bf b}=( \alpha,\alpha\beta, \alpha\beta,\alpha,\alpha, \alpha\beta)$. Then ${\bf b}$ is a {\Pl} weighte vector for $k=2,n=4$ following Example \ref{ex_pl_wgt_vec}. Define a Pl\"{u}cker permutation $\sigma$ on $\{0,1,2,3,4,5\}$ by $\sigma(0)=5$, $\sigma(5)=0$ and $\sigma(i)=i$ for $i=1,2,3,4$. Then $\G_{\bf b}(2,4)$ is a divisive weighted Grassmann orbifold.
\end{example}
\begin{remark}
Consider a {\Pl} weight vector ${\bf b}=(1,12,12,1,1,12)$ for $k=2, n=4$ corresponding to $W=(5,-6,5,5)$ and $a=2$. Consider the scaled {\Pl} weight vector $7{\bf b}=(7,84,84,7,7,84)$ which corresponds to $W=(41,-36,41,41)$ and $a=2$. Now $\G_{\bf b}(2,4)$ isomorphic to $\G_{7\bf b}(2,4)$ by Lemma \ref{lemma_int_red}. Using Lemma \ref{thm_per_plu_vec}, $\G_{7\bf b}(2,4)$ isomorphic to $\G_{\bf b'}(2,4)$ where ${\bf b'}=(84,84,84,7,7,7)$ that correspond to $W=(80,3,3,3)$ and $a=1$. All the weighted Grassmann orbifolds appearing above are divisive.
\end{remark}
The proof of the following theorem follows from same argument as in \cite[Theorem 3.19]{BS}.
\begin{theorem}\label{thm_div_Gsm}
A divisive weighted Grassmann orbifold has a CW complex structure with even dimensional cells $\{\CC^{d'(\lambda^i)}\}_{i=0}^m$.
\end{theorem}
\begin{proof}
    From \eqref{def_e_lam}, we see that if
$[z]=[(z_0,\dots,z_m)]\in E(\lambda^i),$
then $z_{ j}=0$ for every $j<i$.
The $G(b_i)$-action on $E(\lambda^i )$, described in \eqref{eq_t_act} becomes trivial whenever $b_{i}$  divides $b_j$ for all $\lambda^j\in I(\lambda^i)$. Let $\G_{\bf b}(k,n)$ is a divisive weighted Grassmann orbifold. Consider the $q$-CW complex structure on $\G_{\sigma\bf b}(k,n)$ given by $\{\frac{E(\lambda^i)}{G(b_{\sigma (i )})}\}_{i=0}^m$ following Remark \ref{rem_per_cell}. Since $\G_{\bf b}(d,n)$ is divisive, there exists a {\Pl} permutation $\sigma$  such that $b_{\sigma(i)}$ divides $b_{\sigma(j)}$ for every $\lambda^j\in I(\lambda^i)$. Thus the action of $G(b_{\sigma (i )})$ on $E(\lambda^i )$ is trivial. Consequently, $$\frac{E(\lambda^i)}{G(b_{\sigma (i )})}\cong E(\lambda^i)\cong \CC^{d'(\lambda^i)},$$
where $\CC^{d'(\lambda^i)}$ is defined in \eqref{eq_c_d_i}.
\end{proof}

\begin{corollary}\label{cor}
Integral cohomologies of divisive weighted Grassmann orbifolds are torsion-free and concentrated in even degrees.
\end{corollary}

\section{Structure constants for the equivariant cohomology rings of divisive weighted Grassmann orbifolds with integer coefficients}\label{sec_st_con_div_gsm_orb}

In this section, we explore the $T^n$-equivariant cohomology ring of divisive weighted Grassmann orbifolds with integer coefficients. Our main focus lies on studying the multiplication formulae in the integral equivariant cohomology ring of divisive weighted Grassmann orbifolds. We consider the equivariant Schubert basis of the aforementioned ring as a $\ZZ[y_1,\dots,y_n]$ algebra and compute all the structure constants concerning that basis. Moreover, we demonstrate the positivity of the equivariant structure constants.

Let $T^n=(S^1)^n$ be the $n$-dimensional compact abelian torus. Consider the $T^n$-action on $\G_{\bf{b}}(k,n)$ as the restriction of the $(\CC^*)^n$-action defined in \eqref{eq_c_act}. Let $\G_{\bf b}(k,n)$ be a divisive weighted Grassmann orbifold. We can assume that $b_i$ divides $b_{i-1}$ for all $i=1,2,\dots,m$ using Lemma \ref{thm_per_plu_vec}. Define $d_{ij}:=\frac{b_i}{b_j}$ for $i<j$. We recall $I(\lambda^i)$ from \eqref{eq_R_lambda} and define
\begin{equation*}
R(\lambda^i):=\{\lambda^j:\lambda^i \in I(\lambda^j)\}.
\end{equation*}
Then $|R(\lambda^i)|=d(\lambda^i)$. For every Schubert symbol $\lambda=(\lambda_1, \ldots, \lambda_k)$ we define $Y_{\lambda}:=\sum_{i=1}^ky_{\lambda_i}\in \mathbb{Z}[y_1,\dots, y_n]=H^{*}_{T^n}(\{pt\};\ZZ)$.
\begin{theorem}\cite[Theorem 4.7]{BS}\label{thm_int_eq_coh_div_gsm}
 The $T^n$-equivariant cohomology ring of a divisive weighted Grassmann orbifold ${\rm{Gr}}_{\bf{b}}(k,n)$ with integer coefficients is given by
\begin{align*}
H_{T^n}^{*}({\rm{Gr}}_{\bf{b}} (k,n);\mathbb{Z}) = \Big \{\alpha\in \bigoplus_{j=0}^{m}{\mathbb{Z}}[y_1,\dots, y_n]~\big{|}~(Y_{\lambda^i}-d_{ij}Y_{\lambda^j})~\mbox{divides}~(\alpha_j-\alpha_i) \text{ if } \lambda^i\in R(\lambda^j)\Big \}.  
\end{align*}
\end{theorem}

Next, we give a combinatorial description of the equivariant Schubert classes ${\bf b}\widetilde{\mathbb{S}}_{\lambda^i}\in H_{T^n}^{*}({\rm{Gr}}_{\bf{b}} (k,n);\mathbb{Z})$ corresponding to every Schubert symbol $\lambda^i$.

\begin{definition}\cite{BS}\label{prop_eq_Sc_cls}
An element ${\bf{b}}\widetilde{\mathbb{S}}_{\lambda^i}\in H_{T^n}^{*}(\G_{\bf{b}} (k,n);\mathbb{Z})$ is called an equivariant Schubert class corresponding to the Schubert symbol $\lambda^i$ if the following conditions are satisfied.
\begin{enumerate}
	\item  ${\bf{b}}\widetilde{\mathbb{S}}_{\lambda^i}|_{\lambda^j}\neq 0  \implies \lambda^j \succeq \lambda^i$. 
	\item  ${\bf{b}}\widetilde{\mathbb{S}}_{\lambda^i}|_{\lambda^i}=\prod_{\lambda^\ell\in R(\lambda^i)}(Y_{\lambda^\ell}-\dfrac{b_\ell}{b_i}Y_{\lambda^i})$.
	\item ${\bf{b}}\widetilde{\mathbb{S}}_{\lambda^i}|_{\lambda^j}$ is a homogeneous polynomial of $y_1,y_2,\dots,y_n$ of degree $d({\lambda^i})$.
\end{enumerate}
\end{definition}
The next Proposition follows from \cite[Proposition 5.3]{BS} as well as \cite[Proposition 4.3]{HHH}.
\begin{proposition}
    There exists a unique Schubert class ${\bf{b}}\widetilde{\mathbb{S}}_{\lambda^i}\in H_{T^n}^{*}(\G_{\bf{b}} (k,n);\mathbb{Z})$ corresponding to every Schubert symbol $\lambda^i$. Moreover $\{{\bf{b}}\widetilde{\mathbb{S}}_{\lambda^i}\}_{i=0}^m$ forms a basis for the equivariant cohomology ring $ H_{T^n}^{*}(\G_{\bf{b}} (k,n);\mathbb{Z})$ as a $\ZZ[y_1,\dots,y_n]$ algebra.
\end{proposition}

Subsequently, we have the following product formulae in the equivariant cohomology ring of divisive weighted Grassmann orbifold with integer coefficients.     
\begin{equation}\label{eq_wgt_mul}
{\bf{b}}\widetilde{\mathbb{S}}_{\lambda^i}{\bf{b}}\widetilde{\mathbb{S}}_{\lambda^j}=\sum_{\ell=0}^m {\bf{b}}\widetilde{\mathscr{C}}_{i~j}^{\ell}{\bf{b}}\widetilde{\mathbb{S}}_{\lambda^\ell}
\end{equation}

The author and Sarkar have studied some properties of ${\bf{b}}\widetilde{\mathscr{C}}_{i~j}^{\ell}$ in \cite[Lemma 5.5]{BS}. If ${\bf b}=(1,1,\dots,1)$, we denote the basis $\{{\bf{b}}\widetilde{\mathbb{S}}_{\lambda^i}\}_{i=0}^m$ by  $\{\widetilde{\mathbb{S}}_{\lambda^i}\}_{i=0}^m$, which forms a basis for the equivariant cohomology ring $H^*_{T^n}(\G(k,n);\ZZ)$ as a $\ZZ[y_1,\dots,y_n]$ algebra. There exists $\widetilde{\mathscr{C}}_{i~j}^\ell\in \ZZ[y_1,\dots,y_n]$ such that
\begin{equation}\label{eq_mul}
\widetilde{\mathbb{S}}_{\lambda^i}\widetilde{\mathbb{S}}_{\lambda^j}=\sum_{\ell=0}^m\widetilde{\mathscr{C}}_{i~j}^\ell\widetilde{\mathbb{S}}_{\lambda^\ell}
\end{equation} 

 Knutson-Tao found a nice combinatorial formulae of the equivariant structure constants $\widetilde{\mathscr{C}}_{i~j}^\ell$ in terms of puzzles in \cite[Section 2]{KnTa}. Corresponding to every Schubert symbol $\lambda=(\lambda_1,\dots,\lambda_k)$, one can consider a string consisting of $k$ ones and $(n-k)$ zeroes in the following order. One for each $(n+1-\lambda_i)$-th position and zero for the remaining $(n-k)$ positions. In this way, we can identify the set of all Schubert symbols $\lambda=(\lambda_1,\dots,\lambda_k)$ and the set of strings consisting of $k$ ones and $(n-k)$ zeroes in arbitrary order, as explored by Knutson and Tao, see \cite{KnTa}.

Moreover, from \cite[Section 1.2]{KnTa} it follows that the equivariant structure constants $\widetilde{\mathscr{C}}_{i~j}^{\ell}$ defined in \eqref{eq_mul} can be expressed as a polynomial in $\{(y_r-y_{r-1}):1\leq r\leq n-1\}$ such that coefficients are non-negative integers. 
Let $R$ be a finite collection of elements in $\{1,2,\dots,n-1\}$. Define $u_r:=y_r-y_{r+1}$ for $r\in \{1,2,\dots,n-1\}$ and $U_R=\prod_{r\in R}u_r$. Then we can rewrite the above equivariant structure coefficient by the following.
\begin{equation}\label{thm_eq_st_con}
\widetilde{\mathscr{C}}_{i~j}^{\ell}=\sum_{|R|=d(\lambda^i)+d(\lambda^j)-d(\lambda^\ell)}K(i,j,\ell,R)U_R
\end{equation}
where $K(i,j,\ell,R)\in \ZZ_{\geq0}$ is the coefficient of $U_R$ in $\widetilde{\mathscr{C}}_{i~j}^{\ell}$. Thus  $K(i,j,\ell,R)=0$ if there is no monomial in the expansion of  $\widetilde{\mathscr{C}}_{i~j}^{\ell}$ which is constant multiple of $U_R$.

\begin{remark}
\begin{enumerate}
\item The elements in $R$ need not be always distinct; any elements may appear finitely many times. For instance, $U_R$ could be $u_3^2u_4^3$ corresponding to $R=\{3,3,4,4,4\}$.
    \item The sign convention in the equivariant structure coefficient described in \eqref{thm_eq_st_con} differs from the formulae appearing in \cite{KnTa}. This discrepancy arises from a permutation
$\sigma _r:\{ 1,2,\dots ,n\} \rightarrow \{ 1,2,\dots ,n\}$  defined by
$\sigma _r(i)=n+1-i$.
The same permutation acts on the polynomial ring $\mathbb{Z}[y_1,\dots ,y_n]$ via
$\sigma _r(y_i)=y_{n+1-i}$, and therefore induces an action on $H_{T^n}^*(\mathrm{G}(k,n);\mathbb{Z})$.
If $\{ S_{\lambda^i}\}_{i=0}^m$ denotes the basis of
$H_{T^n}^*(\G(k,n);\mathbb{Z})$ used in \cite{KnTa}, then the two conventions are related by $\widetilde{\mathbb{S}}_{\lambda^i}=\sigma_r(\widetilde{S}_{\lambda^i}).$
\end{enumerate}
\end{remark}

We recall $(\CC^*)^{n+1}$-action on $Pl(k,n)$ from \eqref{eq_c_action} and the subgroup $WD$ of $(\CC^*)^{n+1}$ from \eqref{eq_wd}.
The $T^{n+1}$ action on $Pl(k,n)$ induces an $T_{\bf b}$ action on $\G_{\bf{b}}(k,n)$, where  $T_{\bf b}:=\frac{T^{n+1}}{WD}$. We denote $H^*(BT^{n+1};\QQ)=\QQ[y_1,\dots,y_n,z]$ and $H^{*}(BT_{{\bf{b}}};\QQ)=\QQ[x_1,\dots,x_n]$, where $x_r=y_r-\frac{w_r}{a}z$ for $1\leq r \leq n$.

There is an isomorphism between the two equivariant cohomology ring
$$\pi^{*}:H_{T^n}^{*}(\G(k,n);\QQ)\to H^{*}_{T^{n+1}}(Pl(k,n);\QQ)$$ as algebras over  $H^{*}(BT^n;\QQ)=\QQ[y_1,\dots,y_n]$. Moreover,
The advantage of taking the $T_{\bf{b}}$ action on $\G_{\bf{b}}(k,n)$ is 
$$\pi^{*}_{{\bf{b}}}:H_{T_{\bf{b}}}^{*}(\G_{\bf{b}}(k,n);\QQ)\to H^{*}_{T^{n+1}}(Pl(k,n);\QQ)$$
is isomorphism as algebras over $H^{*}(BT_{\bf{b}};\QQ)=\QQ[x_1,\dots,x_n]$, see \cite[Proposition 3.1]{AbMa1}. In addition there exists an isomorphism $$\phi^{*}:H^{*}(BT_{\bf{b}};\QQ)\to H^{*}(BT^n;\QQ)$$ defined by $\phi^{*}(x_i)=y_i$. Moreover, $\phi^*$ induce an isomorphism
$$H_{T_{\bf{b}}}^{*}(\G_{\bf{b}} (k,n);\QQ)\cong H_{T^n}^{*}(\G_{\bf{b}} (k,n);\QQ).$$ 
This follows immediately from the explicit descriptions of the corresponding GKM rings given in
Theorem \ref{thm_cohom_rig} and \cite[Theorem 4.1]{AbMa1} respectively.

\begin{figure}
    \centering
    \begin{tikzpicture}

\node (A) at (-.5,3) {$H^*_{T^n}(\mathrm{Gr}(k,n);\mathbb{Q})$};
\node (B) at (9,3) {$H^*_{T^{n+1}}(\mathbb{P}\ell(k,n);\mathbb{Q})$};

\node (C) at (-.5,0) {$H^*_{T^n}(\mathrm{Gr}_b(k,n);\mathbb{Q})$};
\node (D) at (9,0) {$H^*_{T_b}(\mathrm{Gr}_b(k,n);\mathbb{Q})$};

\node at (4.5,2.7) {$\pi^*$};
\node at (8.5,1.5) {$\pi_b^*$};
\node at (4.5,0.3) {$\phi^*$};
\node at (4.5,-0.3) {$y_i\longleftarrow x_i$};

\node at (4,3.3) {$\mathbb{Q}[y_1,\dots,y_n]\text{-algebra isomorphism}$};
\node at (10.8,1.5) {$\mathbb{Q}[x_1,\dots,x_n]\text{-algebra}$};
\node at (10.6,1) {isomorphissm};
\node at (0.8,1.5) {${\bf b}=(1,\dots,1)$};

\draw[->] (A) -- (B);
\draw[->] (D) -- (C);

\draw[->] (C) -- (A);
\draw[->] (D) -- (B);

\end{tikzpicture}
    \caption{A diagram explaining the algebra isomorphisms}
    \label{eq_comm_dia}
\end{figure}

\begin{proposition}\cite[Theorem 4.2]{AbMa1}\label{thm_pl_cohom}
The $T^{n+1}$-equivariant cohomology ring of $Pl(k,n)$ with rational coefficients is given by 
\begin{align*}
H_{T^{n+1}}^{*}(Pl(k,n);\mathbb{Q})= \Big \{\alpha\in \bigoplus_{i=0}^m \mathbb{Q}[y_1,\dots, y_n,z]/(Y_{\lambda^i}+z)~\big{|}~&(\alpha_{i}-\alpha_{j})\in (Y_{\lambda^i}+z,Y_{\lambda^j}+z)\\ &   \text{ if } |\lambda^i\cap \lambda^j|=k-1 \Big \}.  
	\end{align*}
\end{proposition}

For every $i\in \{0,1\dots,m\}$, define $P\widetilde{\mathbb{S}}_{\lambda^i}\in H^{*}_{T^{n+1}}(Pl(k,n);\QQ)$ by 
\begin{equation}\label{eq_pi}
P\widetilde{\mathbb{S}}_{\lambda^i}:=\pi^{*}(\widetilde{\mathbb{S}}_{\lambda^i}).
\end{equation}

Since $\pi^{*}$ is an algebra isomorphism, $\{P\widetilde{\mathbb{S}}_{\lambda^i}\}_{i=0}^m$  is a basis for $ H^{*}_{T^{n+1}}(Pl(k,n);\QQ)$ as a $\QQ[y_1,\dots,y_n]$-algebra. Moreover, $(\pi^{*}_{{\bf{b}}})^{-1}(P\widetilde{\mathbb{S}}_{\lambda^i})$ forms a basis of $H_{T_{\bf{b}}}^{*}(\G_{\bf{b}} (k,n);\QQ)$ as a $\QQ[x_1,\dots,x_n]$-algebra. The next Lemma follows from the the diagram in Figure \ref{eq_comm_dia} and the uniqueness of the equivariant Schubert class ${\bf b}\mathbb{S}_{\lambda^i}$.

\begin{lemma}\label{eq_bs_lam}
  \begin{equation*}
{\bf{b}}\widetilde{\mathbb{S}}_{\lambda^i}=\phi^*\circ(\pi^{*}_{{\bf{b}}})^{-1}(P\widetilde{\mathbb{S}}_{\lambda^i}).
\end{equation*}  
\end{lemma}

\begin{proof}
  From Definition \ref{prop_eq_Sc_cls}, it follows  $\widetilde{\mathbb{S}}_{\lambda^i}|_{\lambda^i}=\prod_{\lambda^\ell\in R(\lambda^i)}(Y_{\lambda^\ell}-Y_{\lambda^i})$. 
	Thus $$P\widetilde{\mathbb{S}}_{\lambda^i}|_{\lambda^i}=\prod_{\lambda^\ell\in R(\lambda^i)}(Y_{\lambda^\ell}-Y_{\lambda^i})=\prod_{\lambda^\ell\in R(\lambda^i)}(Y_{\lambda^\ell}+z)\in\mathbb{Q}[y_1,y_2,\dots, y_n,z]/(Y_{\lambda^i}+z).$$	
Since $x_r=y_r-\frac{w_r}{a}z$, we have $X_{\lambda^i}=Y_{\lambda^i}-(\frac{b_i-a}{a})z$. Thus

$$
 \prod_{\lambda^\ell\in R(\lambda^i)}(X_{\lambda^\ell}-\dfrac{b_{\ell}}{b_i}X_{\lambda^i})=\prod_{\lambda^\ell\in R(\lambda^i)}(Y_{\lambda^\ell}+z)-\dfrac{b_{\ell}}{b_i}(Y_{\lambda^i}+z)=\prod_{\lambda^\ell\in R(\lambda^i)}(Y_{\lambda^\ell}+z) \in {\QQ[y_1,\dots,y_n]}/{(Y_{\lambda^i}+z)}.$$
   \begin{align*}
 \text{ Also, } 
 \phi^*\Big(\prod_{\lambda^\ell\in R(\lambda^i)}(X_{\lambda^\ell}-\dfrac{b_{\ell}}{b_i}X_{\lambda^i})\Big)=\prod_{\lambda^\ell\in R(\lambda^i)}(Y_{\lambda^\ell}-\dfrac{b_{\ell}}{b_i}Y_{\lambda^i})=
{\bf{b}}\widetilde{\mathbb{S}}_{\lambda^i}|_{\lambda^i}.
\end{align*}
Since $\pi^{*}$ is an isomorphism, $P\widetilde{\mathbb{S}}_{\lambda^i}|_{\lambda^j}\neq 0  \implies \lambda^j \succeq \lambda^i$. Similarly, $ \phi^*\circ(\pi^{*}_{{\bf{b}}})^{-1}(P\widetilde{\mathbb{S}}_{\lambda^i})$ satisfies the same property as $P\widetilde{\mathbb{S}}_{\lambda^i}$. Thus $P\widetilde{\mathbb{S}}_{\lambda^i}$ satisfies all the assumptions of Definition \ref{prop_eq_Sc_cls}.
Now the proof follows from the uniqueness of ${\bf{b}}\widetilde{\mathbb{S}}_{\lambda^i}$.
\end{proof}

For the sake of computing the equivariant structure constants ${\bf{b}}\widetilde{\mathscr{C}}_{i~j}^{\ell}$, we view the product  $U_R{\cdot}P\widetilde{\mathbb{S}}_{\lambda^q}$ in  $H_{T^{n+1}}^*(Pl(k,n))$ as a $\QQ[x_1,\dots x_n]$ algebra. Note that 
\begin{align}\label{eq_x_alg}
U_R=\prod_{r\in R}(y_r-y_{r+1})=\prod_{r\in R}\Big((x_r-x_{r+1})+(w_r-w_{r+1})\frac{z}{a}\Big)=\sum_{s=0}^{|R|}L_{s,R}(\frac{z}{a})^s.
\end{align}
The coefficients $L_{s,R}$ can be computed by the following, where we consider $S$ as a sub collection of $R$ of cardinality $s$
\begin{equation}\label{eq_Lsqr}
    L_{s,R}=\sum_{S:~|S|=s}\Big(\prod_{r\in S}(w_r-w_{r+1})\prod_{r\in R\setminus S}(x_r-x_{r+1})\Big)
\end{equation}
\begin{remark}\label{rem_1st_coef}
    If $s=|R|$, then $ L_{s,R}=\prod_{r\in R}(w_r-w_{r+1})$.
\end{remark}

For, $k<n$, the smallest Schubert symbol $\lambda^0$ is given by $\lambda^0=(1,2,\dots,k)$. Moreover the Schubert symbol $\lambda^1$ is given by $\lambda^1=(1,2,\dots,k-1,k+1)$.

\begin{proposition}\label{prop_ps_lam}
	$P\widetilde{\mathbb{S}}_{\lambda^1}|_{\lambda^i}=(Y_{\lambda^0}+z)$ for $i=1,2,\dots,m$.
\end{proposition}
\begin{proof}
Using \cite[Lemma 5.6]{BS}, we have 
$$\widetilde{\mathbb{S}}_{\lambda^1}|_{\lambda^i}=(Y_{\lambda^0}-Y_{\lambda^i}).$$
Since, $P\widetilde{\mathbb{S}}_{\lambda^1}=\pi^{*}(\widetilde{\mathbb{S}}_{\lambda^1})$ we have
\[P\widetilde{\mathbb{S}}_{\lambda^1}|_{\lambda^i}=Y_{\lambda^0}-Y_{\lambda^i}=Y_{\lambda^0}+z\in \frac{\QQ[y_1,\dots,y_n,z]}{(Y_{\lambda^i}+z)}.\]
\end{proof}

For any two Schubert symbol $\lambda^i$ and $\lambda^j$, we denote $\lambda^j \to \lambda^i$ if $d(\lambda^j) = d(\lambda^i)+1$
and $\lambda^i\preceq\lambda^j$. Following \cite[Proposition 2]{KnTa}, and Proposition \ref{prop_ps_lam}, the Pieri rule in $H_{T^{n+1}}^*(Pl(k,n))$ as a $\QQ[y_1,\dots,y_n]$ algebra is given by
\begin{align}\label{eq_pro_pl_dn}
(Y_{\lambda^0}+z)P\widetilde{\mathbb{S}}_{\lambda^q}&=(Y_{\lambda^0}-Y_{\lambda^q})P\widetilde{\mathbb{S}}_{\lambda^q}+\sum_{\lambda^\ell:~\lambda^\ell\to \lambda^q} P\widetilde{\mathbb{S}}_{\lambda^\ell}.\nonumber\\
\implies z{\cdot}P\widetilde{\mathbb{S}}_{\lambda^q}&=-Y_{\lambda^q}P\widetilde{\mathbb{S}}_{\lambda^q}+\sum_{\lambda^\ell:~\lambda^\ell\to \lambda^q} P\widetilde{\mathbb{S}}_{\lambda^\ell}.
\end{align}

The ring $H_{T^{n+1}}^*(Pl(k,n))$ can also be considered  as a $\QQ[x_1,\dots,x_n]$ algebra. In order to describe the above product as a $\QQ[x_1,\dots,x_n]$ algebra we substitute $y_r=x_r+\frac{w_r}{a}z$ for all $r=1,2,\dots,n$ in the RHS of \eqref{eq_pro_pl_dn}. Thus
\begin{align}\label{eq_alg_rel}
    &z{\cdot}P\widetilde{\mathbb{S}}_{\lambda^q}=-(X_{\lambda^q}+\frac{b_q-a}{a}z)P\widetilde{\mathbb{S}}_{\lambda^q}+\sum_{\lambda^\ell:~\lambda^\ell\to \lambda^q} P\widetilde{\mathbb{S}}_{\lambda^\ell}\nonumber\\
    &\implies\frac{z}{a}{\cdot}P\widetilde{\mathbb{S}}_{\lambda^q}=-\frac{1}{b_q}X_{\lambda^q} P\widetilde{\mathbb{S}}_{\lambda^q}+\frac{1}{b_q}\sum_{\lambda^\ell:~\lambda^\ell\to \lambda^q} P\widetilde{\mathbb{S}}_{\lambda^\ell}.
\end{align}

\begin{remark}
The explicit relation between the $\QQ[x_1,\dots,x_n]$ algebra structures and $\QQ[y_1,\dots,y_n]$ algebra structures on $H_{T^{n+1}}^*(Pl(k,n))$ follows from \eqref{eq_alg_rel}. More explicitly, for $1\leq r\leq n$     $$y_rP\widetilde{\mathbb{S}}_{\lambda^q}=(x_r+\frac{w_r}{a}z)P\widetilde{\mathbb{S}}_{\lambda^q}=x_rP\widetilde{\mathbb{S}}_{\lambda^q}-\frac{w_r}{b_q}X_{\lambda^q} P\widetilde{\mathbb{S}}_{\lambda^q}+\frac{w_r}{b_q}\sum_{\lambda^\ell:~\lambda^\ell\to \lambda^q} P\widetilde{\mathbb{S}}_{\lambda^\ell}.$$
\end{remark}

 For $s\geq 1$, repeated evaluation of the product in \eqref{eq_alg_rel} leads to the following formula.
\begin{equation*}
    (\frac{z}{a})^sP\widetilde{\mathbb{S}}_{\lambda^q}=\sum_{\lambda^\ell\succeq\lambda^q} {\bf{b}}\widetilde{{C}}_{q~g^s}^{\ell}P\widetilde{\mathbb{S}}_{\lambda^\ell}.
\end{equation*}

The coefficients ${\bf{b}}\widetilde{{C}}_{q~g^s}^{\ell}$ is given by the following:
\begin{equation}\label{eq_rep_wgt_per_rul}
  {\bf{b}}\widetilde{{C}}_{q~g^s}^{\ell} =\sum_{\substack{\lambda^\ell=\lambda^{\ell_0}\to\lambda^{\ell_1} \to\\ \dots\to\lambda^{\ell_r} =\lambda^q}}\frac{1}{b_{\ell_1}b_{\ell_2}\dots b_{\ell_r}} \sum_J(\prod_{t=0}^r (-\frac{1}{b_{\ell_t}}X_{\lambda^{\ell_t}})^{j_t}),  
\end{equation}
where $J$ runs over all sequences $(j_0,j_1,\dots,j_r)$ of non negative integers satisfying $j_0+j_1+\dots+j_r=s-r$.
Note that $r=d(\lambda^\ell)-d(\lambda^q)$. Thus ${\bf{b}}\widetilde{{C}}_{q~g^s}^{\ell}=0$ if $d(\lambda^\ell)>s+d(\lambda^q)$.

\begin{remark}\label{rem_2nd_coef}
    If $d(\lambda^\ell)=d(\lambda^q)+s$ then $${\bf{b}}\widetilde{{C}}_{q~g^s}^{\ell} =\sum_{\substack{\lambda^\ell=\lambda^{\ell_0}\to\lambda^{\ell_1} \to\\ \dots\to\lambda^{\ell_s} =\lambda^q}}\frac{1}{b_{\ell_1}b_{\ell_2}\dots b_{\ell_s}}.$$
\end{remark}

 We define $\mathcal{L}_{s,R}:=\phi^*(L_{s,R})$ and ${\bf{b}}\widetilde{\mathscr{C}}_{q~g^s}^{\ell}:=\phi^*({\bf{b}}\widetilde{{C}}_{q~g^s}^{\ell})$. In other words, if we substitute $x_i$ by $y_i$ in $L_{s,R}$ and ${\bf{b}}\widetilde{{C}}_{q~g^s}^{\ell}$ we get $\mathcal{L}_{s,R}$ and ${\bf{b}}\widetilde{\mathscr{C}}_{q~g^s}^{\ell}$ respectively.

\begin{theorem}\label{thm_wei_st_con}
The equivariant structure constant ${\bf{b}}\widetilde{\mathscr{C}}_{i~j}^{\ell}$ is given by $${\bf{b}}\widetilde{\mathscr{C}}_{i~j}^{\ell}=\sum_{\lambda^\ell\succeq\lambda^q\succeq \lambda^i,\lambda^j}\sum_{R}\sum_{s=0}^{|R|}K(i,j,q,R)\mathcal{L}_{s,R} {\bf{b}}\widetilde{\mathscr{C}}_{q~g^s}^{\ell}.$$
\end{theorem}
\begin{proof}
	Using Lemma \ref{eq_bs_lam} in  \eqref{eq_eq_st_con} together with \eqref{eq_x_alg} we have
	\begin{align}\label{eq_eq_st_con}
{\bf{b}}\widetilde{\mathbb{S}}_{\lambda^i}{\bf{b}}\widetilde{\mathbb{S}}_{\lambda^j}&=\sum_{|R|=d(\lambda^i)+d(\lambda^j)-d(\lambda^q)}K(i,j,q,R)\sum_{s=0}^{|R|}\mathcal{L}_{s,R} (\frac{z}{a})^s{\bf{b}}\widetilde{\mathbb{S}}_{\lambda^q}\nonumber\\
  &=\sum_{|R|=d(\lambda^i)+d(\lambda^j)-d(\lambda^q)}K(i,j,q,R)\sum_{s=0}^{|R|}\mathcal{L}_{s,R} \sum_{\lambda^\ell\succeq\lambda^q}{\bf{b}}\widetilde{\mathscr{C}}_{q~g^s}^{\ell}{\bf{b}}\widetilde{\mathbb{S}}_{\lambda^\ell} \\   &=\sum_{\lambda^\ell}\left(\sum_{\lambda^\ell\succeq\lambda^q\succeq \lambda^i,\lambda^j}\sum_{R}\sum_{s=0}^{|R|}K(i,j,q,R)\mathcal{L}_{s,R}  {\bf{b}}\widetilde{\mathscr{C}}_{q~g^s}^{\ell}\right){\bf{b}}\widetilde{\mathbb{S}}_{\lambda^\ell}\nonumber
	\end{align}
\end{proof}

\begin{corollary}[Equivariant Pieri rule] \cite[Proposition 5.7]{BS}\label{prop_eq_wgt_pie_rul}
$${\bf{b}}\widetilde{\mathbb{S}}_{\lambda^1}{\bf{b}}\widetilde{\mathbb{S}}_{\lambda^i}=(Y_{\lambda^0}-\frac{b_{0}}{b_{i}}Y_{\lambda^i}){\bf{b}}\widetilde{\mathbb{S}}_{\lambda^i}+\frac{b_{0}}{b_{i}}\sum_{\lambda^j: \lambda^j\rightarrow \lambda^i}{\bf{b}}\widetilde{\mathbb{S}}_{\lambda^j}.$$
\end{corollary}

\begin{proof}
 Following \cite[Proposition 2]{KnTa}, and the isomorphism $\pi^*$ the Pieri rule in $H_{T^{n+1}}^*(Pl(k,n))$ is given by
\begin{align}\label{eq_pieri_rule}
P\widetilde{\mathbb{S}}_{\lambda^1}P\widetilde{\mathbb{S}}_{\lambda^i}&=(Y_{\lambda^0}-Y_{\lambda^i})P\widetilde{\mathbb{S}}_{\lambda^i}+\sum_{\lambda^j: \lambda^j\rightarrow \lambda^i}P\widetilde{\mathbb{S}}_{\lambda^j}
\end{align}
To describe the product as a $\QQ[x_1,\dots,x_n]$-algebra we substitute $Y_{\lambda^i}=X_{\lambda^i}+\frac{b_i-a}{a}z$. Thus
\begin{align*}
    (Y_{\lambda^0}-Y_{\lambda^i})P\widetilde{\mathbb{S}}_{\lambda^i}&=(X_{\lambda^0}-X_{\lambda^i}+(b_0-b_i)\frac{z}{a})P\widetilde{\mathbb{S}}_{\lambda^i}\\
     &=(X_{\lambda^0}-X_{\lambda^i})P\widetilde{\mathbb{S}}_{\lambda^i}+\frac{b_0-b_i}{b_i}((-X_{\lambda^i})P\widetilde{\mathbb{S}}_{\lambda^i}+\sum_{\lambda^j: \lambda^j\rightarrow \lambda^i}P\widetilde{\mathbb{S}}_{\lambda^j})\text{ (using }\eqref{eq_alg_rel})\\
     &=(X_{\lambda^0}-\frac{b_0}{b_i}X_{\lambda^i})P\widetilde{\mathbb{S}}_{\lambda^i}+\frac{b_0-b_i}{b_i}\sum_{\lambda^j: \lambda^j\rightarrow \lambda^i} P\widetilde{\mathbb{S}}_{\lambda^j}
 \end{align*}
Then the proof follows from  \eqref{eq_pieri_rule} using Lemma \ref{eq_bs_lam}.
\end{proof}

Next, we prove that 	${\bf{b}}\widetilde{\mathscr{C}}_{i~j}^{\ell}\in \ZZ[y_1,y_2,\dots,y_n]$ for all $i,j$ and $\ell$.
We call a polynomial in $ \ZZ[y_1,y_2,\dots,y_n]$ is primitive if the gcd of all coefficients of the monomials is one.

\begin{theorem}\label{prop_st_con_int}
Let ${\rm Gr}_{\bf{b}}(k,n)$ be a divisive weighted Grassmann orbifold. Then the equivariant structure constants ${\bf{b}}\widetilde{\mathscr{C}}_{i~j}^{\ell}$ are in $ \ZZ[y_1,y_2,\dots,y_n]$ for all $i,j,\ell\in \{0,1,\dots,m\}$.
\end{theorem} 
\begin{proof}
Let $\lambda^q$ be a minimal Schubert symbol such that $\lambda^i \preceq \lambda^q$ and $\lambda^j \preceq \lambda^q$ both holds Here $\lambda^q$ is minimal means there is no other Schubert symbol $\lambda^p$ such that $\lambda^p\prec \lambda^q$ and $\lambda^i \preceq \lambda^p$, $\lambda^j \preceq \lambda^p$ both hold. Therefore, 
 $({\bf{b}}\widetilde{\mathbb{S}}_{\lambda^i}{\bf{b}}\widetilde{\mathbb{S}}_{\lambda^j})|_{\lambda^\ell}\neq0$ implies $\lambda^\ell \succeq \lambda^q$ and consequently ${\bf{b}}\widetilde{\mathscr{C}}_{i~j}^{\ell}\neq 0$ implies $\lambda^\ell \succeq \lambda^q$. Note that if $\lambda^{\ell}\succ \lambda^q$ then ${\bf{b}}\widetilde{\mathbb{S}}_{\lambda^\ell}|_{\lambda^q}=0$. Considering the $\lambda^q$-th component in the both the side of \eqref{eq_wgt_mul} we have 
\begin{equation}\label{eq_mul_d}
({\bf{b}}\widetilde{\mathbb{S}}_{\lambda^i}{\bf{b}}\widetilde{\mathbb{S}}_{\lambda^j})|_{\lambda^q}={\bf{b}}\widetilde{\mathscr{C}}_{i~j}^{q}{\bf{b}}\widetilde{\mathbb{S}}_{\lambda^q}|_{\lambda^q}.   \end{equation}

	Now  $({\bf{b}}\widetilde{\mathbb{S}}_{\lambda^i}{\bf{b}}\widetilde{\mathbb{S}}_{\lambda^j})|_{\lambda^q}\in \ZZ[y_1,y_2,\dots,y_n]$ and $${\bf{b}}\widetilde{\mathbb{S}}_{\lambda^q}|_{\lambda^q}=\prod_{\lambda^\ell \in R(\lambda^q)}(Y_{\lambda^\ell}-\dfrac{b_{\ell}}{b_{q}}Y_{\lambda^q})\in \ZZ[y_1,y_2,\dots,y_n].$$
 Using Theorem \ref{thm_wei_st_con}, ${\bf{b}}\widetilde{\mathscr{C}}_{i~j}^{q}\in \QQ[y_1,y_2,\dots,y_n]$. If ${\bf{b}}\widetilde{\mathscr{C}}_{i~j}^{q}\notin \ZZ[y_1,y_2,\dots,y_n]$, let $d$ be the product of the denominators of the coefficients of the monomials in the expression of ${\bf{b}}\widetilde{\mathscr{C}}_{i~j}^{q}$. Then $\overline{\bf{b}}\widetilde{\mathscr{C}}_{i~j}^{q}:=d{\cdot}{\bf{b}}\widetilde{\mathscr{C}}_{i~j}^{q}\in \ZZ[y_1,\dots,y_n]$.

Now multiplying both sides of \eqref{eq_mul_d} by $d$ we get
\begin{equation}\label{eq_gauss_lem}
    d{\cdot}({\bf{b}}\widetilde{\mathbb{S}}_{\lambda^i}{\bf{b}}\widetilde{\mathbb{S}}_{\lambda^j})|_{\lambda^q}=\overline{\bf{b}}\widetilde{\mathscr{C}}_{i~j}^{q}\prod_{\lambda^\ell \in R(\lambda^q)}(Y_{\lambda^\ell}-\dfrac{b_{\ell}}{b_{q}}Y_{\lambda^q})
\end{equation}

Let $p$ be a prime divisor of $d$. Then reducing \eqref{eq_gauss_lem} modulo $p$, we obtain the following:

$$0=\mathscr{P}_1\mathscr{P}_2 ~~~~(\text{mod}~ p),$$

where $\mathscr{P}_1=\overline{\bf{b}}\widetilde{\mathscr{C}}_{i~j}^{q}~~~~(\text{mod}~ p)$ and $\mathscr{P}_2=\prod_{\lambda^\ell \in R(\lambda^q)}(Y_{\lambda^\ell}-\dfrac{b_{\ell}}{b_{q}}Y_{\lambda^q})~~~~(\text{mod}~ p)$.

Note that $\prod_{\lambda^\ell \in R(\lambda^q)}(Y_{\lambda^\ell}-\dfrac{b_{\ell}}{b_{q}}Y_{\lambda^q})$ is a primitive polynomial. Thus $\mathscr{P}_2\neq 0~~~~(\text{mod}~ p)$. Also $({\ZZ}/{p\ZZ})[y_1,\dots,y_n]$ is an integral domain. Thus $\mathscr{P}_1= 0~~~~(\text{mod}~ p)$. Thus $\frac{1}{p}\overline{\bf{b}}\widetilde{\mathscr{C}}_{i~j}^{q}$ is in $\ZZ[y_1,\dots,y_n]$. In other words, we can cancel a factor of $p$ in $d$  in \eqref{eq_gauss_lem}, and still have an equation in $\ZZ[y_1,\dots,y_n]$. Now using induction on the prime factors of $d$ we get $${\bf{b}}\widetilde{\mathscr{C}}_{i~j}^{q}=\frac{1}{d}\overline{\bf{b}}\widetilde{\mathscr{C}}_{i~j}^{q}\in \ZZ[y_1,\dots,y_n].$$

Next, consider the minimal Schubert symbol $\lambda^r$ such that 
$$({\bf{b}}\widetilde{\mathbb{S}}_{\lambda^i}{\bf{b}}\widetilde{\mathbb{S}}_{\lambda^j}-{\bf{b}}\widetilde{\mathscr{C}}_{i~j}^{q}{\bf{b}}\widetilde{\mathbb{S}}_{\lambda^q})|_{\lambda^r}\neq 0.$$ 
Then we follow the same argument to conclude
$${\bf{b}}\widetilde{\mathscr{C}}_{i~j}^{r}\in\ZZ[y_1,y_2,\dots,y_n].$$  
Now we can complete the proof inductively.
\end{proof}

Next, we prove the positivity condition of the equivariant structure constants ${\bf{b}}\widetilde{\mathscr{C}}_{i~j}^{\ell}$ in the sense of \cite{Gra}.
\begin{theorem}\label{thm_pos_eq_st_con}
The equivariant structure constant ${\bf{b}}\widetilde{\mathscr{C}}_{i~j}^{\ell}$ can be written as a polynomial of $y_{r}-y_{r+1}$ for $r=1,2,\dots,{n-1}$ and $(-Y_{\lambda^0})$ with coefficients are non negative integers. 
\end{theorem}
\begin{proof}
 The authors in \cite{KnTa} proved that $K(i,j,q,R)\geq 0$. For every $r\in \{1,2,\dots,n\}$ there exist Schubert symbol $\lambda^e$ and $\lambda^{f}$ such that $\lambda^e<\lambda^f$ and $w_r-w_{r+1}=b_e-b_f$. Thus for a divisive weighted Grassmann orbifold $w_r-w_{r+1}\geq 0$. From \eqref{eq_Lsqr} it implies that $\mathcal{L}_{s,R}$ is a polynomial of $y_{r}-y_{r+1}$ for $r=1,2,\dots,{n-1}$ with non negative integer coefficients. 
  $$-Y_{\lambda^i}=(Y_{\lambda^0}-Y_{\lambda^i})-Y_{\lambda^0}, \text{ for } i\geq 0.$$
Now $(Y_{\lambda^0}-Y_{\lambda^i})$ can be written as a polynomial in $y_r-y_{r+1}$ for $r\in \{1,2,\dots,n-1\}$. 
Thus  ${\bf{b}}\widetilde{\mathscr{C}}_{q~g^s}^{\ell}$ described in  \eqref{eq_rep_wgt_per_rul} can be written as a polynomial of $y_{r}-y_{r+1}$ for $r=1,2,\dots,{n-1}$ and $(-Y_{\lambda^0})$ with non negative coefficients. Hence the proof. 
\end{proof}

\section{The integral cohomology ring structures of divisive weighted Grassmann orbifolds}\label{sec_int_coh}

In this section, we compute the integral cohomology rings of divisive weighted Grassmann orbifolds. First, we describe the integral cohomology groups of divisive weighted Grassmann orbifolds using Theorem \ref{thm_div_Gsm}. Then, we consider a basis for the cohomology groups and compute the cup product concerning that basis. The proof of the following proposition follows from Theorem \ref{thm_div_Gsm} and Corollary \ref{cor}. 

\begin{proposition}
Let $\G_{\bf{b}}(k,n)$ be a divisive weighted Grassmann orbifold. Then
\begin{align*}
H^j(\G_b(k,n);\ZZ)= 
     \begin{cases}
        \bigoplus_{i:2d(\lambda^i)=j}\ZZ &\quad\text{if}~  j \text { is even }\\
       0 &\quad\text{otherwise}.
     \end{cases}
\end{align*}
\end{proposition}

Let $\varepsilon: \ZZ[y_1,\dots,y_n]\to \ZZ$ be the ring homomorphism defined by $\varepsilon(y_i)=0$. The ordinary cohomology ring $H^*(\G_{\bf b}(k,n);\ZZ)$ can be recovered from the equivariant cohomology ring by base change along $\varepsilon$.
$$H^*(\G_{\bf b}(k,n);\ZZ) \cong H_{T^n}^*(\G_{\bf b}(k,n);\ZZ) \otimes_{\ZZ[y_1,\dots,y_n]} \mathbb{Z}.$$ Geometrically, the map $\varepsilon$ corresponds to trivializing the torus action.
Thus, we obtain the forgetful map 
$$\xi: H_{T^n}^*(\G_{\bf b}(k,n);\ZZ) \longrightarrow H^*(\G_{\bf b}(k,n);\ZZ),$$
which simply sends each $y_i$ to zero. For every Schubert symbol $\lambda^i$, we can define the weighted Schubert class 
$${\bf{b}}{\mathbb{S}}_{\lambda^i}:=\xi({\bf{b}}{\widetilde{\mathbb{S}}}_{\lambda^i}).$$

Geometrically, ${\bf b}\widetilde{\mathbb{S}}_{\lambda^i}$ is the equivariant cohomology class in the integral equivariant cohomology ring of the divisive weighted Grassmann orbifold corresponding to the closure of the cell ${E}(\lambda ^i)$ of codimension $2d(\lambda^i)$.
Therefore, $\{{\bf{b}}{\mathbb{S}}_{\lambda^i}:2d(\lambda^i)=j\}$ forms a basis for the ordinary cohomology group $H^j(\G_{\bf{b}}(k,n);\ZZ)$. Let $\lambda^i$, $\lambda^j$ and $\lambda^\ell$ be three Schubert symbols such that $d(\lambda^i)+d(\lambda^j)=d(\lambda^\ell)$. Then we have the following product formulae in the ordinary cohomology ring of divisive weighted Grassmann orbifold.
\begin{equation}
	{\bf{b}}{\mathbb{S}}_{\lambda^i}{\bf{b}}{\mathbb{S}}_{\lambda^j}=\sum_{\lambda^\ell: d(\lambda^\ell)=d(\lambda^i)+d(\lambda^j) }{\bf{b}}{\mathscr{C}}_{i~j}^{\ell}{\bf{b}}{\mathbb{S}}_{\lambda^\ell}.
\end{equation}

By the third condition of Definition \ref{prop_eq_Sc_cls}, it follows that equivariant structure constants ${\bf{b}}\widetilde{\mathscr{C}}_{i~j}^{\ell}$ is a homogeneous polynomial in $\ZZ[y_1,y_2,\dots,y_n]$ of degree $d(\lambda^i)+d(\lambda^j)-d(\lambda^\ell)$. Applying the forgetful map $\xi$  kills all terms of positive degree in the $y_i$, so only those structure constants with $d(\lambda^i)+d(\lambda^j)=d(\lambda^\ell)$ survives. For the special case $d(\lambda^i)+d(\lambda^j)=d(\lambda^\ell)$ the equivariant structure constant ${\bf{b}}\widetilde{\mathscr{C}}_{i~j}^{\ell}$
are integers and agree with the structure constant ${\bf{b}}{\mathscr{C}}_{i~j}^{\ell}$ in ordinary cohomology. Therefore, Theorem \ref{thm_wei_st_con} yields an explicit formulae for the structure constant ${\bf{b}}{\mathscr{C}}_{i~j}^{\ell}$.
We recall for any two Schubert symbol $\lambda^i$ and $\lambda^j$ we denote $\lambda^j \to \lambda^i$ if $d(\lambda^j) = d(\lambda^i)+1$
and $\lambda^i\preceq\lambda^j$.
We have the following formulae from the Example \ref{prop_eq_wgt_pie_rul}.

\begin{proposition}[weighted Pieri rule] \label{prop_wgt_pr_ru}
	$${\bf{b}}{\mathbb{S}}_{\lambda^1}{\bf{b}}{\mathbb{S}}_{\lambda^i}=\sum_{\lambda^j:\lambda^j\to \lambda^i}\frac{b_0}{b_{i}}{\bf{b}}{\mathbb{S}}_{\lambda^j}$$
	for all $i=1,2,\dots,m$.
\end{proposition}

Define $b(u_r):=b_e-b_f$ for $r\in \{1,2,\dots,n-1\}$ if $y_r-y_{r+1}=Y_{\lambda^e}-Y_{\lambda^f}$. For every finite sub collection $R$ of $\{1,2,\dots,n-1\}$ we define $b(U_R):=\prod_{r\in R}b(u_r)$. Moreover
\begin{equation}\label{eq_st_con_in_b}
\widetilde{\mathscr{C}}_{i~j}^{\ell}(b):=\sum_{|R|=d(\lambda^i)+d(\lambda^j)-d(\lambda^\ell)}K(i,j,\ell,R)b(U_R)
\end{equation}

For every Schubert symbol $\lambda^q$ such that $\lambda^\ell\succ\lambda^q\succeq\lambda^i,\lambda^j$ and for every sequence $\lambda^\ell=\lambda^{\ell_0}\rightarrow \lambda^{\ell_1}\rightarrow \dots\rightarrow \lambda^{\ell_d}=\lambda^q$ we can calculate 
\begin{equation}\label{eq_dij}
\mathscr{D}_{i~j}^{\ell~q}=\sum_{\substack{\lambda^\ell=\lambda^{\ell_0}\rightarrow \lambda^{\ell_1}\rightarrow\\ \dots\rightarrow\lambda^{\ell_d}=\lambda^q}}\frac{\widetilde{\mathscr{C}}_{i~j}^{q}({\bf{b}})}{b_{\ell_1}b_{\ell_2}\dots b_{\ell_d}}.
\end{equation}
Then we have the following formulae for the structure constant ${\bf{b}} \mathscr{C}_{i~j}^{\ell}$ for the ordinary cohomology ring of divisive weighted Grassmann orbifold with integer coefficients.
\begin{theorem}\label{thm_st_con_ord_coh}
The structure constant ${\bf{b}}\mathscr{C}_{i~j}^{\ell}$ is given by $${\bf{b}}\mathscr{C}_{i~j}^{\ell}=\mathscr{C}_{i~j}^{\ell}+\sum_{\lambda^q~:~\lambda^\ell\succ\lambda^q\succeq\lambda^i,\lambda^j}\mathscr{D}_{i~j}^{\ell~q}.$$
\end{theorem}

\begin{proof}
 The equivariant structure constant ${\bf{b}}\widetilde{\mathscr{C}}_{i~j}^{\ell}$ becomes the  structure constant ${\bf{b}}\mathscr{C}_{i~j}^{\ell}$ for the ordinary cohomology if $d(\lambda^i)+d(\lambda^j)=d(\lambda^\ell)$. Thus we consider  $s=|R|=d(\lambda^i)+d(\lambda^j)-d(\lambda^q)=d(\lambda^\ell)-d(\lambda^q)$ in \eqref{eq_eq_st_con}. Define $d:=d(\lambda^\ell)-d(\lambda^q)$.

\begin{align*}
	{\bf{b}}{\mathscr{C}}_{i~j}^{\ell}&=\sum_{\lambda^\ell\succeq\lambda^q\succeq \lambda^i,\lambda^j} \sum_{R:~|R|=d}K(i,j,q,R)\mathcal{L}_{d,R}{\bf{b}}\widetilde{\mathscr{C}}_{q~g^d}^{\ell}\\
	&=	\sum_{\lambda^\ell\succeq\lambda^q\succeq \lambda^i,\lambda^j}\sum_{R:|R|=d}K(i,j,q,R)b(U_R)\sum_{\substack{\lambda^\ell=\lambda^{\ell_0}\to\lambda^{\ell_1}\to\\\dots\to\lambda^{\ell_d} =\lambda^q}} \frac{1}{b_{\ell_1}b_{\ell_2}\dots b_{\ell_d}}\\
	&=\sum_{\lambda^\ell\succeq\lambda^q\succeq \lambda^i,\lambda^j}\sum_{\substack{\lambda^\ell=\lambda^{\ell_0}\to\lambda^{\ell_1}\to\\
 \dots\to\lambda^{\ell_d} =\lambda^q}}\frac{\widetilde{\mathscr{C}}_{i~j}^{q}({\bf{b}})}{b_{\ell_1}b_{\ell_2}\dots b_{\ell_d}}\\
	&=\mathscr{C}_{i~j}^{\ell}+\sum_{\lambda^q~:~\lambda^\ell\succ\lambda^q\succeq \lambda^i,\lambda^j}\mathscr{D}_{i~j}^{\ell~q}
\end{align*}	
We apply Remark \ref{rem_2nd_coef} together with Remark \ref{rem_1st_coef} to derive the second line from the first line.
\end{proof}

Using the above theorem, one can compute the integral cohomology ring of a divisive weighted Grassmann orbifold explicitly.

\begin{remark}
For a divisive $\G_{\bf b}(k,n)$ the structure constants ${\bf{b}}\mathscr{C}_{i~j}^{\ell}\in \ZZ_{\geq 0}$ follows from Theorem \ref{prop_st_con_int} together with Theorem \ref{thm_pos_eq_st_con}.
\end{remark}

In the following example, we compute the integral cohomology ring of a divisive weighted Grassmann orbifold $\G_{\bf{b}}(2,4)$. 

\begin{example}
 Let $\alpha$ and $\beta$ be two positive integers. Consider a Pl\"{u}cker weight vector ${\bf{b}}:=(\alpha\beta,\alpha\beta,\alpha\beta,\alpha,\alpha,\alpha)$ for $k=2, n=4$. Then ${\bf{b}}$ is a divisive  Pl\"{u}cker weight vector and the corresponding weighted Grassmann orbifold  $\G_{\bf{b}}(2,4)$ is a divisive weighted Grassmann orbifold. Corresponding to every Schubert symbol we have $6$ basis element $\{{\bf{b}}\widetilde{\mathbb{S}}_{\lambda^i}\}_{i=0}^5$. Using Proposition \ref{prop_wgt_pr_ru}, it follows that $${\bf{b}}{\mathscr{C}}_{1~i}^{j}=\frac{b_0}{b_i} \text{ if } d(\lambda^j)=d(\lambda^i)+1 \text{ and } \lambda^j\succ \lambda^i.$$
Now we calculate the remaining structure constants. Note that from \cite{KnTa} we have

$$\widetilde{\mathbb{S}}_{\lambda^3}\widetilde{\mathbb{S}}_{\lambda^3}=(Y_{\lambda^0}-Y_{\lambda^3})(Y_{\lambda^2}-Y_{\lambda^4})\widetilde{\mathbb{S}}_{\lambda^3}+(Y_{\lambda^2}-Y_{\lambda^4})\widetilde{\mathbb{S}}_{\lambda^4}+\widetilde{\mathbb{S}}_{\lambda^5}.$$
Therefore
\begin{align*}
{\bf{b}}\mathscr{C}_{3~3}^{5}&=\mathscr{C}_{3~3}^{5}+\mathscr{D}_{3~3}^{5~4}+\mathscr{D}_{3~3}^{5~3}\\
&=1+\frac{b_2-b_4}{b_4}+\frac{(b_0-b_3)(b_2-b_4)}{b_4b_3}\\
&=1+\frac{b_2-b_4}{b_4}(1+\frac{b_0-b_3}{b_3})\\
&=1+\frac{(b_2-b_4)b_0}{b_4b_3}=1+\beta(\beta-1).	
\end{align*}
Similarly one can calculate
$${\bf{b}}\mathscr{C}_{2~2}^{5}=1+\frac{b_0(b_1-b_2)}{b_2b_4}=1+\beta(\beta-1).$$ 
Moreover, using \cite{KnTa}
$$\widetilde{\mathbb{S}}_{\lambda^3}\widetilde{\mathbb{S}}_{\lambda^2}=(Y_{\lambda^0}-Y_{\lambda^4})\widetilde{\mathbb{S}}_{\lambda^4}.$$
Then
\begin{align*}
	{\bf{b}}\mathscr{C}_{2~3}^{5}&=\mathscr{C}_{2~3}^{5}+\mathscr{D}_{2~3}^{5~4}\\
	&=0+\frac{b_0-b_4}{b_4}=\beta-1.
\end{align*}
All the remaining structure constants in the integral cohomology ring of the above divisive $\G_{\bf{b}}(2,4)$ are zero. This gives the integral cohomology ring of the given divisive weighted Grassmann orbifold explicitly.
\end{example}

\section{appendix}

In this appendix, we discuss some properties on integral cohomology of $\G_{\bf b}(2,4)$.
Recall the subcomplexes $X^i$ of $\G_{\bf{b}}(2,4)$ from \eqref{filtration} for $i=0,\dots,5$. Note that $$X^3=\WW P(b_5,b_4,b_3)\sqcup_{\WW P(b_5,b_4)}\WW P(b_5,b_4,b_2). \text{ Define } $$ $$\eta=\frac{\lcm\{b_5,b_4,b_3\}}{\lcm\{b_5,b_4\}} \text{ and }\eta'=\frac{\lcm\{b_5,b_4,b_2\}}{\lcm\{b_5,b_4\}}.$$

\begin{lemma}\label{lem_no_p_tor_x_3}
	$H_i(X^3;\ZZ)$ has no torsion if $i\neq 2$. Moreover if $\gcd\{\eta,\eta'\}=1$ then $H_i(X^3;\ZZ)$ is torsion-free for all $i$.
\end{lemma}
\begin{proof}
	Consider the cofibration $$L'(b_3;(b_5,b_4))\to \WW P(b_5,b_4,b_2) \to X^3.$$
	This induces an exact sequence in homology by the following
 
	\begin{equation}\label{eq_ex_seq_hom}
		\to \widetilde{H}_{i}(L'(b_3;(b_5,b_4)))\to  H_i(\WW P(b_5,b_4,b_2)) \to   H_i(X^3) \to 	\widetilde{H}_{i-1}(L'(b_3;(b_5,b_4))) \to 
	\end{equation}
	For $i=1$ in \eqref{eq_ex_seq_hom} we get
	$$ 0 \to   H_1(X^3) \to \ZZ \to \ZZ \to.$$
 $ H_1(X^3)$ is either $\{0\}$ or a finite abelian group as $X^3$ has no cell of dimension $1$. Also $ H_1(X^3) \to 	\ZZ $ is injective. Thus $ H_1(X^3)=\{0\}$. 
 
For $i=3$ in \eqref{eq_ex_seq_hom}, we get
	$$\to \ZZ \to  0 \to   H_3(X^3) \to 	0 \to \ZZ.$$
	Thus $H_3(X^3)=\{0\}$. For $i=4$ in \eqref{eq_ex_seq_hom} we get
	$$\to 0 \to  \ZZ \to   H_4(X^3) \to 	\ZZ \to 0. $$
	Thus $ H_4(X^3)=\ZZ\oplus \ZZ$.
	For $i\geq 5$ in \eqref{eq_ex_seq_hom} we get
	$$\to 0 \to  0 \to   H_i(X^3) \to 	0 \to .$$
Now put $i=2$ in \eqref{eq_ex_seq_hom} we get
	$$\to 0 \to  \ZZ \to   H_2(X^3) \to \ZZ_\eta \to 0,$$

	Thus, torsion can arise only in $H_2(X^3)$, and there can be maximum $\eta$ torsion. There exist another cofibration $$L'(b_2;(b_5,b_4))\to \WW P(b_5,b_4,b_3) \to X^3.$$ Considering the homology exact sequence corresponding to this cofibration, we can say $H_2(X^3)$ has maximum $\eta'$ torsion. Hence, if $\gcd(\eta,\eta')=1$ then $H_i(X^3,\ZZ)$ has no torsion for all $i$.
\end{proof}

\begin{corollary}
	For a prime $p$, there always exist a Pl\"{u}cker permutation $\sigma$ on $\{0,1,\dots,5\}$ such that the integral cohomology of the 3rd skeleton $X^3$ in $\G_{\sigma \bf b}(2,4)$ is torsion free. 
\end{corollary}

\begin{theorem}\label{tor_in_2,4}
	$H^i({\rm{Gr}}_{\bf{b}}(2,4);\ZZ)$ is torsion free for any weighted Grassmann orbifold ${\rm{Gr}}_{\bf b}(2,4)$ if $i\neq 3$.
\end{theorem}
\begin{proof}
	For a prime $p$, choose a Pl\"{u}cker permutation $\sigma$ such that the minimum power of $p$ attains at $b_{\sigma(0)}$. Note that there are $4$ different Pl\"{u}cker permutation which satisfy the above condition. We can choose $\sigma$ such that $ r_{\sigma(1)} \leq \min\{r_{\sigma(1)},  r_{\sigma(2)}, r_{\sigma(3)}, r_{\sigma(4)}\}$, where $p^{r_i}$ is the $p$-th content in $b_i$. Thus using Lemma \ref{thm_per_plu_vec}, we can assume that $r_0\leq r_1\leq \min\{r_{1},  r_{2}, r_{3}, r_{4}\}$. From the cofibration $L^{'}(b_{0};{\bf{b}}^{(0)})\to X^4 \hookrightarrow \G_{\bf{b}}(2,4) $ we have
	$$\to  H_i(X^4) \to H_i(\G_{\bf{b}}(2,4)) \to \widetilde{H}_{i-1}(L'(b_{0};{\bf{b}}^{(0)})) \to  H_{i-1}( X^4)\to .$$
	For $i=8$, we have $$\to 0\to H_8(\G_{\bf{b}}(2,4))\to \ZZ\to 0\to .$$
	For $2\leq i\leq 7$,  $\widetilde{H}_{i-1}(L'(b_{0};{\bf{b}}^{(0)}))$ has no $p$-torsion. Thus $H_i(\G_{\bf{b}}(2,4))$ has no $p$-torsion if the same is true for $H_i( X^4)$.
	Similarly from the cofibration
	$L^{'}(b_{1},{\bf{b}}^{(1)})\to  X^3\hookrightarrow X^4 $ we have
	$$ \to \widetilde{H}_{i}(L'(b_{1};{\bf{b}}^{(1)}))\to  H_i( X^3) \to   H_i( X^4) \to \widetilde{H}_{i-1}(L'(b_{1};{\bf{b}}^{(1)})) \to   .$$ 
	For $i=6$ we 	
	$$ 0\to  H_6( X^4) \to \ZZ\to 0  .$$
	Thus $ H_6( X^4)=\ZZ$. For  $2\leq i \leq 5$,  $H_i( X^4)$ has no $p$-torsion if the same is true for $H_i( X^3)$ as $\widetilde{H}_{i-1}(L'(b_{1};{\bf{b}}^{(1)}))$ has no $p$-torsion. Using Lemma \ref{lem_no_p_tor_x_3}, $H_i( X^3)$ has no torsion if $i\neq 2$. 
	Thus $H_i(\G_{\bf{b}}(2,4);\ZZ)$ has no torsion if $i\neq 2$. Hence, the result follows the universal coefficient theorem for cohomology.
\end{proof}

\begin{theorem}
   Let ${\bf{b}}=(b_0,b_1,b_2,b_3,b_4,b_5)\in (\ZZ_{\geq 1})^6$ be  a Pl\"{u}cker weight vector for $2<4$. Then $H^*(\G_{\bf{b}}(2,4);\ZZ)$ has no torsion if $b_1=b_2$ and $(b_0,b_5)=1$. 
\end{theorem}
\begin{proof}
    Since ${\bf{b}}$ is a Pl\"{u}cker weight vector we have $b_1+b_4=b_2+b_3$. If $b_1=b_2$ then $b_3=b_4$. Thus ${\bf{b}}=(b_0,b_1,b_1,b_3,b_3,b_5)$. For a prime $p$, we can assume $p$-th content of $b_1$ divides $p$-th content of $b_3$. Otherwise consider the Pl\"{u}cker permutation $\sigma$ defined by $\sigma(1)=4,\sigma(2)=3$, $\sigma(4)=1,\sigma(3)=2$ and $\sigma(i)=i$ for $i\in \{0,5\}$. Then $\sigma {\bf{b}}=(b_0,b_3,b_3,b_1,b_1,b_5)$. Then by Theorem \ref{thm_no_p_tor}, $H^*(X^4;\ZZ)$ has no torsion. Considering the cohomology exact sequence in \eqref{eq_ex_seq_hom_wgt_gsm}, $H^*(\G_{\bf{b}}(2,4);\ZZ)$ has maximum $b_0$ torsion. Consider the Pl\"{u}cker permutation $\sigma'$ defined by $\sigma'(0)=5$, $\sigma'(5)=0$ and $\sigma'(i)=i$ for $i\in \{1,2,3,4\}$. Then $\sigma' {\bf{b}}=(b_5,b_1,b_1,b_3,b_3,b_0)$. By a similar argument $H^*(\G_{\sigma' {\bf{b}}}(2,4);\ZZ)$ has maximum $b_5$ torsion. Since $(b_0,b_5)=1$ then $H^*(\G_{\bf{b}}(2,4);\ZZ)$ has no torsion.
\end{proof}

{\bf Acknowledgement.} 
The author thanks to Soumen Sarkar for his guidance and support. We have many valuable discussions. The author thanks Hiraku Abe and Tommo Matsumura for helpful discussions. The author would like to thank IIT Madras for the Ph.D. fellowship.

\section{Declarations}
\subsection*{Ethical Approval}
This declaration is not applicable.

\subsection*{Consent to participate}
This declaration is not applicable.

\subsection*{Consent to publish}
This declaration is not applicable.

\subsection*{Competing interests}
The author declares that he has no competing interests.

\subsection*{Funding}
 The author is funded by the Indian Institute of Technology Madras PhD fellowship during this work.

\subsection*{Availability of data and materials}
Data sharing does not apply to this article, as no data sets were generated or analyzed during the current study.

\bibliographystyle{abbrv}
\bibliography{ref-Wedge1}

\end{document}